\pdfoutput=1 
\RequirePackage{fixltx2e}
\documentclass[10pt]{article}
\usepackage[utf8]{inputenc}
\usepackage[T1]{fontenc}
\usepackage{amsmath, amsfonts, amsthm, amssymb, amscd, aliascnt}
\usepackage{a4wide}
\usepackage[normalem]{ulem}
\usepackage{relsize}
\usepackage{cancel}

\usepackage{xr} 

\usepackage[mathcal]{euscript} 
\usepackage{mathrsfs} 
\usepackage{accents, verbatim}
\DeclareMathAlphabet\mathbfcal{OMS}{cmsy}{b}{n}

\input{epsf}
\usepackage{pdfpages, enumerate, graphicx, xcolor, array, booktabs, microtype}
\usepackage{fancyhdr}
\usepackage[nottoc]{tocbibind}
  
\usepackage[pdfusetitle]{hyperref}
\hypersetup{linktoc = all}
\hypersetup{hidelinks}
\hypersetup{bookmarksnumbered}

\numberwithin{equation}{section}
\numberwithin{table}{section}
\numberwithin{figure}{section}

\usepackage{footnotebackref}

\theoremstyle{plain}
        \newtheorem{theorem}{Theorem}[section]  

        \newtheorem{proposition}[theorem]{Proposition}

        \newtheorem{lemma}[theorem]{Lemma}
        \newtheorem{corollary}[theorem]{Corollary}
        \newtheorem{remark}[theorem]{Remark}

\numberwithin{equation}{section}
\let\oldmarginpar\marginpar
\renewcommand\marginpar[1]{\- \oldmarginpar[\raggedleft\footnotesize #1]%
{\raggedright\footnotesize #1}}
\newcommand \Lbb{\mathbb L}

\newcommand \Acal {\mathcal A}
\newcommand \Bcal {\mathcal B}
\newcommand \Ccal {\mathcal C}

\newcommand \bei {\begin{itemize}}
\newcommand \eei {\end{itemize}}

\newcommand \eps \epsilon
\newcommand \Ewavek {E_{\text{W}}^{k}}
\newcommand \Ewaveak {E_{\text{W}}^{a,k}}
\newcommand \Nwaveak {N_{\text{W}}^{a,k}}
\newcommand \Nwaveaktwo {N_{\text{W}}^{a,k+2}}
\newcommand \Eoneak {E_r^{a,k}}
\newcommand \Eonea {E_r^a}
\newcommand \EoneaN {E_r^{a,N}}
\newcommand \EtwoaN {E_{s}^{a,N}}

\newcommand \Etwoak {E_{s}^{a,k}}
\newcommand \Etwoa {E_{s}^a}
\newcommand \Etwozero {E^{0}_{s}}
\newcommand \Etonezero {E^{0}_{r}}

\newcommand \Eoneaktwo {E_r^{a,k+2}}
\newcommand \Etwoaktwo {E_{s}^{a,k+2}}

\newcommand \Ewavea {E_{\text{W}}^a}
\newcommand \Ewave {E_{\text{W}}}
\newcommand \EKG {E_\text{KG}}

\newcommand \Kint {\mathcal{K}_{\text{int}}}
\newcommand \Kwave {\mathcal{K}_{\text{W}}}

\newcommand \baru {\underline u}
\newcommand \Kcal {\mathcal K}
\newcommand \be {\begin{equation}}
\newcommand \ee {\end{equation}}
\newcommand \bel {\be \label}
\newcommand \del \partial
\newcommand \RR{\mathbb R}
\newcommand \diag {\text{diag}}
\newcommand \Hcal {\mathcal H}

\newcommand{\lp}[2]{\Vert \, #1 \, \Vert_{#2}}

\newcommand{\td}{\widetilde}

\newcommand{\snabla}{ {\, {\slash\!\!\! \nabla  }} }

\newcommand{\jap}[1]{ {\langle #1 \rangle} }

\newcommand \delu {\underline{\del}} 
 \renewcommand \Phi{\widetilde\phi}
 \newcommand \Vol  {dV}

\usepackage{perpage}
\MakePerPage{footnote} 
\begin{document}
 
\title{Boundedness of the conformal hyperboloidal energy\\ for a wave-Klein-Gordon model}

\author{Philippe G. LeFloch\texorpdfstring{$^a$},, 
Jes\'us Oliver\texorpdfstring{$^b$},, 
and Yoshio Tsutsumi\texorpdfstring{$^c$}, 
{}}

\date{October 2022}

\maketitle
\footnotetext[1]{Laboratoire Jacques-Louis Lions \& Centre National de la Recherche Scientifique,
Sorbonne Universit\'e,
4 Place Jussieu, 75252 Paris, France. Email: {\tt contact@philippelefloch.org}.
\\
\indent 
$^b$ Department of Mathematics, California State University, East Bay, 
25800 Carlos Bee Blvd., Hayward, CA 94542, USA.
Email: {\tt jesus.oliver@csueastbay.edu}
\\
\indent $^b$ Institute of Liberal Arts and Sciences,
Kyoto University, 
Yoshida Nihonmatsu-cho, 
Sakyo-ku,
Kyoto 606--8501, Japan.
Email: {\tt tsutsumi@math.kyoto-u.ac.jp}
}

\begin{abstract} 
\mbox{} 
We consider the global evolution problem for a model which couples together a nonlinear wave equation and a nonlinear Klein-Gordon equation, and was independently introduced by LeFloch and Y. Ma and by Q. Wang. 
By revisiting the Hyperboloidal Foliation Method, we establish that a weighted energy of the solutions remains (almost) bounded for all times. The new ingredient in the proof is a hierarchy of fractional Morawetz energy estimates (for the wave component of the system) which is defined from two conformal transformations. The optimal case for these energy estimates 
corresponds to using the scaling vector field as a multiplier for the wave component.
\end{abstract}

  
\setcounter{secnumdepth}{3}


\section{Introduction} 
\label{section1}

\paragraph{Aim of this paper.}

We are interested in the global-in-time existence problem for coupled systems of wave-Klein-Gordon  equations and we study a typical system that arises as a simple model formally derived from the Einstein equations of general relativity in presence of massive 
matter fields. We are interested in spatially compactly supported solutions and in analyzing their decay in time, 
and therefore our analysis focus on the interior of a light cone region. 
Throughout, we work in Minkowski spacetime $\RR^{3+1} :=(\RR^{4},\eta)$ where 
$\eta := \diag(-1,1,1,1)$ denotes the Minkowski metric in inertial coordinates denoted by
$(x^0, x^1, x^2, x^3) = (t,x) = (t,x^i) \in \RR^4$ ($i=1,2,3$). 
The level sets of the time function $t : \RR^{3+1} \mapsto \RR$ are uniformly spacelike and provide natural hypersurfaces on which the initial value problem can be posed. On the other hand, an alternative standpoint is provided by the slices (with $s \geq s_0>1$ for notational convenience) 
\bel{equa-hyperfol} 
\Hcal_s := \big\{(t, x) \, / \, t >0; \, t^2 - r^2 = s^2 \big\}, 
\ee
in which we used the notation $r^2 := |x|^2 = (x^1)^2 + (x^2)^2 + (x^3)^2$. These slices define a foliation by spacelike hypersurfaces of (a subset of) the 
interior of the future light cone from the origin. 

In the present paper, we build upon pioneering works using the hyperbolic foliation \eqref{equa-hyperfol} 
together with earlier works on the global existence problems for wave equations \cite{Hormander,Klainerman87}
and Klein-Gordon equations \cite{Georgiev,Katayama12a,Katayama12b,Klainerman85,LSoffer,LSogge,OzawaTsutayaTsutsumi,Tataru1}. 
Specifically, we follow the presentation in the monograph \cite{PLF-YM-book} by LeFloch and Ma who coined the word ``hyperboloidal foliation method" to denote a strategy which leads to global existence results for coupled wave-Klein-Gordon systems possibly on a curved geometry, and combines several ingredients motivated by the earlier works cited above: 
\bei

\item[(i)] a notion of {\sl weighted functional norm} based on the Lorentz boosts, 

\item[(ii)] a notion of {\sl semi-hyperboloidal frame}  
containing tangent vectors to the hyperboloids, 

\item[(iii)] {\sl weighted Sobolev} and {\sl Hardy inequalities} on hyperboloids, 

\item[(iv)] {\sl sharp pointwise} estimates in sup-norm for wave and Klein-Gordon equations, and 

\item[(v)] nonlinear coupling arguments based on a {\sl hierarchy} enjoyed by the wave-Klein-Gordon system.
\eei
\noindent 
The method in  \cite{PLF-YM-book} overcomes the difficulty that solutions to wave and Klein-Gordon equations have different decay in time, and that the scaling vector fields does not commute with the Klein-Gordon operators, so that the natural choice of weighted norms is too small for applying the standard vector field techniques. This approach was later applied and generalized in \cite{PLF-YM-1} where the spatial decay of solutions is investigated. This line of research was motivated by the study of the Einstein equations in the regime near the (vacuum) Minkowski spacetime. 


\paragraph{The system of interest.}

Specifically, we study here the following coupled system consisting of two wave and Klein-Gordon equations, referred to as 
the LeFloch-Ma-Wang model \cite{PLF-YM-1,Wang1}: 
\bel{eq main}
\aligned
& - \Box u = P^{\alpha\beta}\del_\alpha v\del_\beta v + Rv^2,
\qquad 
 - \Box v + c^2 v  =  u  H^{\alpha\beta} \del_\alpha\del_\beta v,
\endaligned
\ee
in which $\Box := - \del_t^2 + \Delta_x$ denotes the wave operator on $\RR^{3+1}$ 
and $P^{\alpha\beta}$, $H^{\alpha\beta}$, $R$, and $c \neq 0$ are constants.
The unknowns $u, v$  are scalar fields defined over Minkowski space $\RR^{3+1}$, 
and \eqref{eq main} is supplemented with initial data
$u_0, u_1, v_0, v_1$ prescribed on the hyperboloid $\Hcal_{s_0} \subset \RR^{3+1}$:  
\bel{initialdata}
\aligned
& u|_{\Hcal_{s_0}} = u_0, \quad && \del_t u|_{\Hcal_{s_0}} = u_1,
\qquad 
& v|_{\Hcal_{s_0}} = v_0, \qquad && \del_t v|_{\Hcal_{s_0}} = v_1. 
\endaligned
\ee 
Provided these data are sufficiently small and regular, the existence of a global-in-time solution $(u,v)$
to \eqref{eq main}- \eqref{initialdata} was established in \cite{PLF-YM-1}. The same problem was studied by Wang~\cite{Wang1} and Ionescu and Pausader \cite{IonescuPausader0} in a different functional framework.  

Our objective in the present paper is to revisit their result in a {\sl fractional weighted energy norm,} motivated from the Morawetz estimate and the conformal energy. Namely, we will rely on two conformal transformations of the equations under consideration and, in turn, in comparison to \cite{PLF-YM-1} we will establish an improved time decay estimate, namely the (almost) {\sl boundedness of the energy} for all times. In other words, we improve upon the conclusion of LeFloch-Ma's theorem, but at the expense of imposing a stronger spatial decay on the initial data. Our proof builds upon recent work by Oliver and Sterbenz \cite{OliverSterbenz} and Oliver \cite{Oliver}, which provided an extension of the standard vector field method. The issue of the boundedness of the energy was also recently established for a variety of models in 
\cite{DIPP,DLLei,PLF-CW}. 

\paragraph{Motivations from the Einstein equations.}

Recall that, in \cite{PLF-YM-1}, the model \eqref{eq main} was ``extracted'' from the Einstein the equations of general relativity since it provided a simple, yet non-trivial, example of coupling between a wave equation and a Klein-Gordon equation. For a discussion of the Einstein equations, we refer to \cite{PLF-YM-3}--\cite{PLF-YM-Third--Part} as well as \cite{Bigorgne2,IonescuPausader1,Wang1}. Furthermore, recall that a different strategy was proposed by Ionescu and Pausader \cite{IonescuPausader1}, who also successfully treated \eqref{eq main} by using Fourier resonance techniques and also imposed a condition on  the $1/2$-derivatives of the wave component of the data. 
(Their regularity condition is reminiscent of our condition with the exponent $a=1/2$; see below.) For further results concerning the nonlinear coupling between wave and Klein-Gordon equations, we refer to \cite{KriegerSterbenz,OhTataru} and the references cited therein.  For other important related developments, see \cite{FJS3,HintzVasy2,HuneauStingo,IfrimStingo}. 


\paragraph{Notation and outline.}

Throughout, we write $A \lesssim  B$ (resp. ``$A\ll B$"; ``$A\approx B$") if $A \ \leqslant \ C B$ for some fixed $C>0$ which may change from line to line (resp. $A\leqslant \eps B$ for a small $\eps$; both $A\lesssim B$ and $B\lesssim A$). Our notation for the range of the indices is as follows: 
$\alpha,\beta, \ldots = 0,1,2,3$, while $i,j, \ldots = 1,2,3$ and $b,c \ldots = 1,2$.
We make use of Einstein summation convention throughout, and always rely on the Minkowski metric $m$ to raise or lower indices. The notation $m=m^{-1} = \diag(-1,1,1,1)$ stands for the Minkowski metric in $(t,x^i)$ coordinates.

An outline of this paper is as follows. In Section~\ref{section2}, we present our global energy bounds for the model, as well as an estimate for the wave equation which is of interest for its own sake. We proceed by deriving, in  Section~\ref{section3}, a general conformal energy identity satisfied by solutions to the wave equation. In Section~\ref{section4}, we then introduce Bondi coordinates and establish several tensorial identities. This leads us, in  Section~\ref{section5}, to our main fractional Morawetz energy estimate. In Section~\ref{section6}, the framework is finally applied to the model of interest. 


\section{Statement of the main results} 
\label{section2}

\subsection{The hyperboloidal foliation}

\paragraph{Spacetime slices.}

We need some additional notation and, by following  \cite{PLF-YM-1}, we introduce the interior of the future light cone with vertex $(1,0,0,0)$, namely 
\be
\Kcal : = \big\{ (t,x^i ) \, / \,  r < t-1 \big\},
\ee
while we foliate the interior of the cone $\big\{ (t,x^i) \, /  r<t \big\}$ from the origin $(0,0,0,0)$
by the spacelike hypersurfaces of constant hyperboloidal time $s$
\be
\Hcal_s := \big\{ (t,x^i ) \, / \, t>0, \,   t^{2}- r^{2} =s^{2 } \big\}. 
\ee
We have fixed some $s_0>1$ and for any pair $s_0 \leq s_1 < s_2$, we introduce the following sub-domain of $\Kcal$ bounded by two hyperboloids: 
\be
\Kcal_{[s_1,s_2]} := \big\{ (t,x^i ) \, / \, s_1^{2}\leq t^{2}- r^{2}\leq s_2^{2} \ , r<t-1 \big\} 
\subset \Kcal.
\ee
Within $\Kcal$ we single out the \textit{interior domain} 
\be
\Kint := \big\{ (t,x^i ) \, / \,  r <  t/2 \big\} \ \subset \Kcal.
\ee
More precisely, the \textit{wave domain} (i.e. essentially the support of the solutions to the wave equation considered below) 
is defined to be the complement in
$\Kcal$ of the domain $\Kint$, namely 
\be
\Kwave := \big\{ (t,x^i ) \, / \, (t/2)  \leq r \leq t-1 \big\} \subset \Kcal.
\ee

\paragraph{Vector fields.}

We define the \textit{semi-hyperboloidal frame} $\delu = (\delu_\alpha) =(\underline \del_0,\delu_1,\delu_2,\delu_3)$ to be
\bel{semih_frame}
\delu_{0} := \del_t, 
\qquad
\delu_{i} := \del_{i}+ \frac{x^i }{t}\del_t,
\ee
and we also set $\delu_{x} :=(\delu_1,\delu_2,\delu_3)$. Observe that the vector fields $\delu_i$ generate the tangent
space to the hyperboloids of constant $s$. We also further decompose
Span$\{\delu_i \}_{i=1,2,3}$ into radial and angular components 
$(\underline\partial_r,\snabla)$ with
\be
\delu_{r} := \del_r+ {r \over t} \del_t,
\ee
while $\snabla$ denotes derivatives tangential to the surfaces with constant $t-r$ and $r$.


\subsection{Functional spaces of interest}

\paragraph{Standard energy norms.}

On any spacelike hypersurface $\Hcal_s$ of constant hyperboloidal radius $s= \sqrt{t^{2}- r^{2}}$,
we use the $L^2$ norm 
\be
\lp{\phi (s)}{L^2(\Hcal_s)} 
: = \Big(\int_{\Hcal_s }\phi^{2} \, dx \Big)^{1/2},
\qquad s \geq s_0,
\ee
where $dx=dx^1 dx^2 dx^3$ is the standard Euclidean volume form. Observe that $dx$ differs from the natural volume form on hyperboloids (which is $\Vol =(s/t) dx$); the above notation will turn out to be convenient in our analysis. 
On any given hypersurface $\Hcal_s$, the standard \textit{hyperboloidal energy} of wave and Klein-Gordon fields, respectively, read\footnote{after normalization of the Klein-Gordon mass. } 
\begin{align}
& \lp{\phi}{\Ewave} 
: = \lp{\frac{s}{t}\del_t \phi}{L^2(\Hcal_s)} + \lp{ \underline \del_{x}\phi}{L^2(\Hcal_s)},  
&& \phi  = \phi(s), 
\label{wave_stenergy}
\\
& \lp{\psi}{\EKG} 
:= 
\lp{\frac{s}{t}\del_t \psi}{L^2(\Hcal_s)}+
\lp{ \underline \del_{x}\psi}{L^2(\Hcal_s)}
+ \lp{\psi}{L^2(\Hcal_s)}, 
&& \psi = \psi(s). 
\label{energy_kg}
\end{align}


\paragraph{Weighted energy norms.}

Next, let us define the so-called optical functions\footnote{It will be clear from the context whether $u$ denotes the Minkowski optical function $t- r$ or the wave component of \eqref{eq main}.}
$
u := t- r$
and  
$\baru := t+r,$
which is now used in defining our weights as follows: 
\be
\tau_- := \jap{u}, \qquad
\tau_+ := \jap{\baru}, \qquad
\tau_0 := \tau_-/\tau_+.
\ee
Observe that $\tau_-$ can be regarded as a distance to the light cone, while $\tau_+$ provides a decay weight in, both, timelike and spacelike directions. 
Here, for any function $f$ we use the notation $\jap{f} =(1+f^{2})^{1/2}$ for the so-called Japanese bracket. 
The, given an arbitrary exponent $a \in [0,1]$, we intrduce the \textit{weighted hyperboloidal energy} of the wave component $\phi= \phi(s)$ by
\be
\lp{\phi}{\Ewavea}
:= 
\lp{ \big(1+a\tau_+^a \tau_0^{\max(a, 1/2)} \big)  \frac{s}{t}\del_t \phi}{L^2(\Hcal_s)}+
\lp{s^a  \underline \del_{x}\phi}{L^2(\Hcal_s)}
+a \, \lp{\frac{s^a}{t}\phi}{L^2(\Hcal_s)},
\label{frac_energy}
\ee
where, by definition, $E_{W}^{0} =E_{W}$. 
In addition, to deal with a source term $F$ in the wave equation, we will use a spacetime weighted norm denoted by the subscript $N_{W}^a $, given by 
\be
\lp{F}{N_{W}^a [s_0,s_1]}  
: =
\int_{s_0}^{s_1}\lp{\tau_{+}^{a} \frac{s}{t} \, F(s)}{L^2(\Hcal_s)} \, ds.
\ee


\paragraph{High-order norms.}

We now define versions of the  norms which incorporate additional decay based on the translations and Lorentz boosts of Minkowski space. In the $(t,x^i)$ coordinates, we define the Lie algebra of translations and boosts as follows (with $\alpha=0,1,2,3$ and $i=1,2,3$): 
\bel{def_vfields}
\Lbb  
:=  \big\{ \del_{\alpha} \big\} \cup \big\{ L_{i} = t\del_{i}+x^i \del_t   \big\}.
\ee  
Given any multi-indices $I=(I_\alpha)$ and $J=(J_i)$, we use the standard notation
$\del^{I} = \del^{I_0}_0 \del^{I_1}_1\del^{I_2}_2 \del^{I_3}_3$ and 
$L^J  = L^{J_1}_1 L^{J_2}_2 L^{J_3}_3$, and we also set 
$|I| := I_0+I_1+I_2+I_3$ and $|J| := J_1+J_2+J_3$. 
We define the {\it high-order energy norms} of interest on any given hyperboloid $\Hcal_s$ as 
\be 
\lp{\phi}{\Ewaveak}   :=   \sum_{|I|+|J|\leqslant k} \lp{\del^{I}L^J  \phi}{\Ewavea},   
\quad\qquad \phi = \phi(s)
\ee
as well as, for a source term in the wave equation, the spacetime norm
\be
\lp{F}{N^{a,k} [s_0,s_1]}   :=  \sum_{|I|+|J|\leqslant k}
\lp{\del^{I}L^J  F}{N_{W}^a [s_0,s_1]}. 
\ee
Similar definitions are used the remaining norms. 


\paragraph{Conformal norms.}

Finally, let us introduce the Bondi-like derivative operators
$
\del^B_{u}:= \del_t$
and $ 
\del^B_{r}:= \del_t+ \del_r$,
which are well-defined except at the center $r=0$. We will work with two choices of conformal weight denoted by $\Omega$, namely 
\be 
\Omega_{r} := r, 
\qquad 
\Omega_{s} := s^{2} = u\baru. \label{weights_definition}
\ee
For each parameter value $a \in [0,1]$, we associate the 
\textit{modified weighted energy} of a function $\phi = \phi(s)$ in Bondi coordinates defined by
\bel{conj_conf_en1}
 \lp{r\phi}{ \Eonea}
: =  \lp{\big( 1+ a\tau_+^a  \tau_0^{\max(a,1/2)} \big) 
 \Big( 
\frac{s}{t}\frac{\del^B_u(r\phi)}{r}, \snabla\phi \Big)}{L^2(\Hcal_s)} + \lp{\tau_+^a  \big(\frac{\del^B_{r}
 (r\phi)}{r},\frac{s}{t}\snabla\phi \big)}{L^2(\Hcal_s)}
\ee
for the choice of weight $\Omega_s$, while for the choice $\Omega_r$ we set 
\bel{conj_conf_en2}
 \lp{s^2 \phi}{\Etwoa}  
:=  \lp{\big( 1+s^{a-1}\tau_- \big)
 \Big(
\frac{s}{t}\frac{\del^B_u(s^2\phi)}{s^{2}}, \snabla\phi \Big)}{L^2(\Hcal_s)} + \lp{s^{a-1}\tau_+ \big(\frac{\del^B_{r}
 (s^2\phi)}{s^2},\frac{s}{t}\snabla\phi \big)}{L^2(\Hcal_s)}. 
\ee 
Observe that  a computation using Hardy's inequality shows the inequalities 
%
\be
\aligned
& \lp{\phi(s)}{\Ewave}\lesssim \lp{r\phi(s)}{\Eonea},
\qquad \lp{\phi(s)}{\Ewave}\lesssim \lp{s^2\phi(s)}{\Etwoa},
\\
& \lp{\phi(s)}{\Ewave}\approx \lp{r\phi(s)}{\Etonezero},
\qquad \lp{\phi(s)}{\Ewave}\approx \lp{s^2 \phi(s)}{\Etwozero}. 
\endaligned
\ee


\subsection{Existence theory for the wave-Klein-Gordon model}

\paragraph{Main statement.}

We are now in a position to state our main result. 

\begin{theorem}[Global energy hierarchy for the wave-Klein-Gordon model]
\label{thm-main}
Consider the coupled nonlinear wave and Klein-Gordon system \eqref{eq main} for some given constants 
$P^{\alpha\beta}, R, H^{\alpha\beta}$, and $c>0$. Then for any integer $N \geq 8$, there exist positive constants $\eps= \eps(N) \in (0,1)$,  $\delta= \delta(N, \eps) \in (0,1)$, and $C_1>0$ such that the initial value problem
\eqref{eq main}- \eqref{initialdata} enjoys the following property. 
For all initial data satisfying, for some $a \in [0,1]$, the smallness and decay conditions 
\bel{eq:main-initial}
\aligned
& 
 \| u(s_0) \|_{E_{W}^N }
+
{
a\lp{ ru(s_0)}{\EoneaN}
+a \, \lp{ s^2 u(s_0)}{\EtwoaN}
}
+
 \| v(s_0) \|_{E_{KG}^N }\leq \eps,
\endaligned
\ee
%
%
the Cauchy problem  \eqref{eq main}- \eqref{initialdata}
admits a global-in-time solution $(u,v)$ defined in the future of
$\Hcal_{s_0}$ and satisfying the energy bounds for all $s \geq s_0$ and any $k = |J| = 0, \ldots, N$ and $|I|+|J|\leq N$
\bel{eq:main}
\aligned
&  
\|\del^{I} L^J u(s)
\|_{\Ewavea} \lesssim
\eps s^{\max(a-1/2,0) \delta {
+ {\max(k, 1) \delta}}},
\\
&  
\| \del^{I} L^J v(s) \|_{\EKG}
\lesssim \eps s^{\max(-a +1/2,0) + k\delta}. 
\endaligned
\ee
Furthermore, the following pointwise decay properties hold (for all $s \geq s_0$ and some constants $C_k$)
\bel{eq:main_pointwise}
\aligned
& \sup_{|I|\leq k}\|  \tau^{1/2}_+ \del^{I} u(s)
\|_{L^\infty(\Hcal_s)}
\lesssim \eps s^{k\delta-1/2}, 
\\
& \sup_{|I|\leq k}\|  \tau^{3/2}_+ \del^{I} v(s) \|_{L^\infty(\Hcal_s)}
\lesssim \eps s^{\max(-a+1/2,0)\delta + k\delta}.
\endaligned
\ee
\end{theorem}

The conclusion of main interest is obtained by {\sl choosing $a=1/2$} in the above theorem, since in this case the exponent $-a+1/2$ cancels out and we arrive at the almost sharp growth $s^{k\delta}$ for {\sl both} the wave and the Klein-Gordon components.
Compare this (almost) energy boundedness statement with the previous result in \cite{PLF-YM-1} which had a
rate of growth of $s^{1/2+k \delta}$ for the energy of the Klein-Gordon component. 
Hence, under stronger decay conditions in comparison to \cite{PLF-YM-1} we arrive at a stronger conclusion as far as the decay of solutions is concerned.


\paragraph{Main energy and pointwise estimates.} 

In the course of our analysis, relying on the weighted spaces defined earlier in this section we will study solution to the wave equation
\bel{equa-Cauchy-wave-eq}  
\aligned
& \Box\phi = F, 
\qquad 
 (\phi, \del_t\phi)|_{\Hcal_{s_0}} = (\phi_0, \phi_1), 
\endaligned
\ee
where the data $(\phi_0,\phi_1)$ are prescribed on the initial hyperboloid $\Hcal_{s_0}$.

\begin{theorem}[Fractional Morawetz estimate in Bondi coordinates]
\label{theor-main_estimate}
Given any $a \in [0,1]$, any sufficiently regular and compactly supported solution $\phi$ to the Cauchy problem \eqref{equa-Cauchy-wave-eq} satisfies ($k \geq 0$ being an arbitrary integer) 
satisfies 
\be
\sup_{s_0 \leq s \leq s_1} 
\big(\lp{\phi(s)}{\Eoneak}+ \lp{\phi(s)}{\Etwoak} \big)
 \lesssim  \lp{\phi(s_0)}{\Eoneak}+ \lp{\phi(s_0)}{\Etwoak}
+ \lp{F}{\Nwaveak [s_0,s_1]}. \label{frac_conf_wave1}
\ee
\end{theorem}

As a consequence, we have the following energy bounds in terms of the semi-hyperboloidal frame. These, in turn, yield a pointwise decay estimate.

\begin{corollary}
\label{theor-main_cor} 
Given any $a \in [0,1]$, any sufficiently regular and compactly supported solution $\phi$ to the Cauchy problem \eqref{equa-Cauchy-wave-eq} satisfies the following properties 
($k \geq 0$ being an arbitrary integer). 
\bei

\item {\bf Weighted energy estimate:} 
\bel{frac_conf_wave2}
\sup_{s_0 \leq s \leq s_1} \lp{\phi(s)}{\Ewaveak}
\lesssim   
\lp{\phi(s_0)}{\Eoneak}+ \lp{\phi(s_0)}{\Etwoak}
+ \lp{\phi(s_0)}{\Ewavek}+ \lp{F}{\Nwaveak [s_0,s_1]}.
\ee

\item {\bf Weighted pointwise decay:} 
\bel{point1}
\aligned
& \sup_{s_0 \leq s\leq s_1}\sum_{|I|+|J|\leq k}
\lp{s^a \tau_+^{1/2}\del^{I}L^J  \phi}{L^\infty(\Hcal_s)}
  \lesssim 
\lp{\phi(s_0)}{\Eoneaktwo}+ \lp{\phi(s_0)}{\Etwoaktwo}
 + \lp{\phi(s_0)}{E_{W}^{k+2}}+
\lp{F}{\Nwaveaktwo [s_0,s_1]}.  
\endaligned
\ee
\eei
\end{corollary}

For $a=1/2$ we obtain the relevant estimate for our purpose in the present paper and this is the key in proving the 
(essential) 
and the weighted energy boundedness stated in Theorem~\ref{thm-main} follows by following the same lines as in \cite{PLF-YM-book}. 
Hence, our main task for the rest of this paper is giving a proof of Theorem~\ref{theor-main_estimate}. 
Note in passing that plugging in the extremal value $a=1$ above allows us to recover
(based on the standard vector field arguments) the $L^\infty$ decay for the linear wave equation in Minkowski space, namely 
$t^{-3/2}$ in timelike regions and $t^{-1}$ in the vicinity of the light cone $t=r$.
  
  
\section{A conformal multiplier identity for the hyperboloidal foliation} 
\label{section3}

\subsection{Energy formalism for the wave equation}

\paragraph{Divergence identity.}

To any solution $\phi$ to the equation \eqref{equa-Cauchy-wave-eq}, we associate the energy-momentum tensor associated with the Minkowski metric $\eta$ and an arbitrary function $\phi$, namely 
\be\label{em_tensor1}
Q_{\alpha\beta}[\phi] = \del_{\alpha}\phi\del_{\beta}\phi  - 
\frac{1}{2}\eta_{\alpha\beta}
\del^\gamma 
\phi  \del_{\gamma}\phi.
\ee
It is well-known that this symmetric two-tensor is related to
the wave operator by the identities 
\bel{divergence1}
\del^\gamma  \big( Q_{\gamma\beta}[\phi]  \big) 
=  F \, \del_{\beta}\phi. 
\ee
To each vector field $X=X^{\gamma} \del_{\gamma}$, we can then associate the one-form field
\be
{}^{(X)}\!P_{\alpha}[\phi]  :=  Q_{\alpha\gamma}[\phi] \, X^{\gamma}.
\label{def_P}
\ee
which, sometimes, is simply denoted by $P_{\alpha}$ when the choice made for $X$ is obvious. 

We also define the symmetric two-tensor ${}^{(X)}\pi = \mathcal{L}_{X}\eta$ as being the Lie derivative of the metric $\eta$ with respect to $X$. It is called  the \textit{deformation tensor} of $X$ and, in local coordinates, one has
\be
{}^{(X)}\pi_{\alpha\beta}
:= 
(\mathcal{L}_{X}\eta)_{\alpha\beta} 
 =  \del_{\alpha}X_{\beta}+ \del_{\beta}X_{\alpha}.
\ee
When the choice of the vector field is obvious from context we simply write
$\pi$ instead of ${}^{(X)}\pi $. Taking the divergence in \eqref{def_P} and using the identities \eqref{divergence1}, we find 
\be
\del^\gamma  \big( P_{\gamma}[\phi] \big)
 = 
F  X\phi +  \frac{1}{2}{}^{(X)} \pi^{\alpha\beta}
Q_{\alpha\beta}[\phi]. \label{stokes_1}
\ee

\paragraph{Integrating between two hyperboloids.}

For any $s_1\geq s_0 >1$, integrating \eqref{stokes_1} over the spacetime domain
$\Kcal_{[s_0,s_1]}$ and using
Stokes' theorem we get the \textit{multiplier identity} associated
with the vector field $X$
\be
\int_{s=s_0} {}^{(X)} P_{\alpha}[\phi]  \, N^\alpha  
\Vol    - \int_{s=s_1} \!{}^{(X)}\!P_{\alpha}[\phi]  \, N^\alpha  
\Vol  =  \int_{s_0}^{s_1}\!\!\!\int_{s=s'}
\Big( F  X\phi  +  \frac{1}{2}{}^{(X)}
\pi^{\alpha\beta}Q_{\alpha\beta} [\phi]  \big)  \Vol ds',
\label{divergence2}
\ee
where  
 $\Vol $ denotes the volume element on hyperboloids: $\Vol = \frac{s}{t}dx$.
Here, $N$ denotes the future-directed unit normal vector to the level sets $s=const$
\bel{normalv}
N^\alpha = \frac{-m^{\alpha\beta} \del_{\beta}s}{(-m^{\alpha\beta}
\del_{\alpha}s\del_{\beta}s)^{1/2}}   
= -m^{\alpha 0}\frac{t}{s} + m^{\alpha r}\frac{r}{s}.   
\ee
We also define the {\it rescaled normal vector} $N'$ via
\be
N{'}^{ \alpha} 
= \frac{s}{t}N^\alpha =-m^{\alpha 0} + m^{\alpha r}{r \over t}. \label{resc_normalv}
\ee
In terms of $N{'}$ we have the identity
\be\label{rescaleddx}
{}^{(X)}P_{\alpha}N^\alpha  \,
\Vol ={}^{(X)}P_{\alpha}N{'}^{ \alpha}   \,
dx. 
\ee
Last, we refer to the integrand on the left hand side 
of the equation \eqref{divergence2} as the \textit{energy density} associated with $X$
with respect to the foliation by spacelike hypersurfaces $s=const$. 


\subsection{Multipliers and commutators}
\label{form_comm}

The following setup will be useful for
constructing weighted energies.
Given a vector field $X$, we define the \textit{traceless deformation tensor} of $X$ to be
$
{}^{(X)} \widehat{\pi} := 
{}^{(X)} \pi - \frac{1}{2}\eta   (\text{Tr }{}^{(X)}\pi).
$
 
\begin{lemma}[Basic formulas involving ${}^{(X)} \widehat{\pi}$]\label{basic_iden_lemma}
Let $\phi$ be a solution to the equation \eqref{equa-Cauchy-wave-eq} and $X$ be a smooth vector field. The 
following identities hold:
\be\label{pi_hat_formula}
{}^{(X)} \widehat{\pi}^{\alpha\beta}  =  -X(\eta^{\alpha\beta})
- \eta^{\alpha\beta }\del_\gamma X^\gamma
+ \eta^{\alpha\gamma}\del_\gamma X^\beta
+ \eta^{\beta\gamma}\del_\gamma X^\alpha, 
\ee
\be\label{divergence_formula}
\del^\gamma  P_{\gamma}  = 
F  X\phi +  \frac{1}{2}\widehat{\pi}^{\alpha\beta}
\del_\alpha\phi \del_\beta\phi 
= 
F  X\phi +  \frac{1}{2}\pi^{\alpha\beta}
Q_{\alpha\beta}, 
\ee
where \eqref{pi_hat_formula} is expressed in local coordinates.
\end{lemma}

\begin{proof} To derive the identity \eqref{pi_hat_formula}, we write 
$
{}^{(X)}\pi^{\alpha\beta} = -X(\eta^{\alpha\beta})
+ \eta^{\alpha\gamma}\del_\gamma X^\beta + 
\eta^{\beta\gamma}\del_\gamma X^\alpha 
$
in local coordinates.
Subtracting the expression $\frac{1}{2} \eta^{\alpha\beta}  (\text{Tr }\pi)
=  \eta^{\alpha\beta}\del_{\gamma}X^{\gamma}$
from both sides completes the proof.
The identity \eqref{divergence_formula} follows directly from
the definition of $Q_{\alpha\beta}[\phi]$.
\end{proof}


\subsection{Conformal transformations for the wave equation}
\label{conf_chan_mult}

Let $g = (g_{\alpha\beta})$ be a Lorentzian metric on an $(3+1)$-dimensional spacetime. Given any positive weight  $\Omega: \RR^{3+1} \to \RR_+$ (assumed to be smooth throughout this paper), we consider the conformal metric $\td{g}_{\alpha\beta}$ defined by 
$
\td{g} := \Omega^{-2}g,
$
and we denote by $\td \nabla$ and $\Box_{\td{g}} = \td \nabla^\alpha\td \nabla_\alpha$ the corresponding covariant derivative and wave operator, respectively. A standard calculation (see \cite{Oliver}) shows that the function 
$
\Phi : =  \Omega\phi
$
 satisfies the Klein-Gordon equation
\bel{trans_eq}
\Box_{\td{g}}\Phi + V\Phi = \Omega^3  F,
\qquad 
V := \Omega^3 \Box_g\Omega^{-1}, 
\qquad F := \Box_{g}\phi. 
\ee 
Furthermore, given any smooth and positive weight $\Lambda:   \RR^{3+1} \to \RR_+$ and in view of the equation \eqref{trans_eq}, it is natural to introduce the \textit{conformal energy-momentum tensor} associated with the function $\Phi$  defined as 
\bel{em_tens_cut} 
\td{Q}_{\alpha\beta}[\Phi, \Lambda] = \Big(\del_\alpha\Phi\del_\beta\Phi  
- 
\frac{1}{2}\td{g}_{\alpha\beta}(\td{g}^{\mu\nu}\del_\mu
\Phi\del_\nu\Phi -  V\Phi^2)\Big)  \Lambda.
\ee
This tensor satisfies the divergence identity
\be
\td{\nabla}^\alpha \td{Q}_{\alpha\beta}[\Phi, \Lambda] = \Lambda \Omega^3  F
\del_\beta\Phi
 +   \frac{\Lambda}{2}(\del_\beta V)
\Phi^2+  \frac{\Omega^{2}}{\Lambda} (\nabla^{\alpha}\Lambda)\td{Q}_{\alpha\beta}.
 \label{div_law_conf}
\ee
In fact, our weights may vanish at the center $r=0$, and in such a case the above identities
are only valid for $r>0$. 
From the identity \eqref{div_law_conf} we have
$$
\aligned
&\td{\nabla}^\alpha X^{\beta}(\td{Q}_{\alpha\beta}[\Phi, \Lambda])
-(\td{\nabla}^\alpha X^{\beta}) \td{Q}_{\alpha\beta}[\Phi, \Lambda]
\\ & = \Lambda \Omega^3  F
X^{\beta}\del_\beta\Phi
 +   \frac{\Lambda}{2}(X^{\beta}\del_\beta V)
\Phi^2
+  \frac{\Omega^{2}}{\Lambda} (\nabla^{\alpha}\Lambda)X^{\beta}\td{Q}_{\alpha\beta}, \label{div_law_conf}
\endaligned
$$
and we have arrived at the following result.

\begin{proposition}
Let $X$ be a vector field associated with the energy flux ${}^{(X)} \td{P}_\alpha  = \td{Q}_{\alpha\beta}
X^\beta$. Then the one-form field ${}^{(X)} \td{P}_\alpha$ satisfies the following
\textit{conformal multiplier identity} associated with the vector field $X$ and the weight $\Omega$
\bel{div_identity1_cut} 
\int_{ s=s_0}\!\! {}^{(X)} \td{P}_\alpha N^\alpha  \Omega^{-2} \ \Vol   - \int_{ s=s_1} \!\!
{}^{(X)} \td{P}_\alpha N^\alpha  \Omega^{-2} \, \Vol  
=   \int_{s_0}^{s_1}\!\!\!\int_{s=s'} 
 \big(\td{\nabla}^\alpha {}^{(X)} \td{P}_\alpha \big)\Omega^{-4} \,  \Vol ds', 
\ee
in which the normal vector field $N$ is given by
\eqref{normalv} and the divergence term on the right-hand side reads
\bel{td_div_iden_cut3} 
\big(\td \nabla^\alpha {}^{(X)}\!\td{P}_\alpha\big) \Omega^{-4} 
= \Lambda F   \Omega^{-1}X\Phi + \Lambda \Omega^{-2} \Acal^{\alpha\beta}
\del_{\alpha}\Phi\del_{\beta}\Phi + \Lambda \Bcal\td \phi^2
+ \Omega^{-2} \Lambda^{-1}\Ccal^{\alpha}{}^{(X)}
\td P_{\alpha}  \ ,
\ee
with $\Acal, \Bcal, \Ccal$ given by 
\bel{AB_formulas1}
\Acal^{\alpha\beta} : =  \frac{1}{2}  \big({}^{(X)} \widehat{\pi} + 2X\ln (\Omega) g^{-1} \big), 
\qquad 
\Bcal :=  \frac{1}{2\Omega^{4}} \big(  X(V)- (\text{Tr }\Acal) V\big),  
\qquad 
\Ccal^{\alpha} := \nabla^{\alpha}\Lambda.
\ee
\end{proposition}

%

\section{Tensor identities in Bondi coordinates} 
\label{section4}

\subsection{Basic definitions}

As shown in Oliver \cite{Oliver} and Oliver and Sterbenz \cite{OliverSterbenz}, the Bondi coordinate system provides one with a very convenient way to establish energy estimates for the wave equation. In establishing our multiplier estimates, 
while we work here with a different, hyperboloidal, foliation, we can still follow the outline of the arguments in \cite{Oliver,OliverSterbenz}. Recall that $u=t- r$ is the Minkowski optical function whose level sets are outgoing null hypersurfaces and that $x^i $, $1\leq i \leq 3$ are the standard Cartesian coordinates. The set of coordinates $(u,x^i)$ is referred to as the {\textit{Bondi coordinate system}} for Minkowski space. We also define the \textit{polar Bondi coordinates} by $(u,r,x^A)$, where locally we choose $x^A$ to be two members of $\widehat x^i=r^{-1}x^i$. We also let $|\snabla^B \phi|^2$ denote the angular gradient of 
$\phi$ with respect to the spheres $u=const$ and $r=const$.

We denote by $\eta$ the Minkowski metric
in Bondi coordinates. In (polar) Bondi coordinates, the components
of $\eta$ have the form
\be
\eta_{uu} =  -1, \qquad
\eta_{ur}  =  -1, \qquad
\eta_{rr} = 0, \qquad
\eta_{bc}  =  r^{2}\delta_{bc},
\label{metric_bondi1}
\ee
with $\delta_{bc}$ denoting the standard metric on 2-spheres.
We denote by $\eta^{-1}$ the inverse Minkowski metric in polar Bondi coordinates, with 
$\eta^{uu} =  0$, $\eta^{ui}  =  - \omega^i, $
and
$\eta^{ij}  =  \delta^{ij}.$
In particular $\det(\eta^{-1})=-1$. Thus, the wave operator $\Box_\eta$ in Bondi coordinates takes the form
\be
\Box_\eta =   -2\del^B_u \del^B_r 
+ (\del^B_{r})^2 - 2r^{-1}\del^B_u + 2r^{-1} \del^B_{r} 
+ r^{-2}\sum_{i<j}(\Omega_{ij})^2.
\label{wave_bondi}
\ee
Henceforth we let $\del=(\del_t,\del_x)$ to denote the $(t,x)$ coordinate derivatives,
and likewise $\del^B=({\del_u^B},\del_x^B)$ denote the $(u,x)$-coordinate derivatives. By the chain rule, we have immediately 
\begin{align}
\del^B_u  =  \del_t, \qquad \del^B_{i}  =  
\del_{i}+ \omega^{i}\del_t, \qquad \del^B_{r}  =  
\del_r+ \del_t   \label{coords_chain_rule}.
\end{align}


\subsection{Various identities and estimates in Bondi coordinates}

Next we compute the key quantities from the equations \eqref{td_div_iden_cut3}, \eqref{AB_formulas1}. The proof of the following two lemmas is postponed to Appendix~\ref{Append-A}. 

\begin{lemma}[Identity for deformation tensor]
\label{moved-to-Append-A1}
Let $\Omega \geq 0$ be a weight and $X=X^{\alpha}\del_{\alpha}$
be a vector field in Bondi coordinates. The contravariant tensor $\Acal$
in \eqref{AB_formulas1} satisfies (at points where $\Omega>0$)
\begin{align}
\Acal(X,\Omega, \eta) & = \frac{1}{2} \big({}^{(X)} \widehat{\pi}
+ 2 X\ln (\Omega) \eta^{-1} \big) 
 \label{bondi_deften_A1}    \\
& = - \frac{1}{2}\big( \mathcal{L}_X \eta^{-1} +
\Big(\del^B_u X^u + \del^B_{r} X^r + \del^B_i\overline X^i 
+2(\frac{X^{r}}{r}-X \ln \Omega) \Big)\eta^{-1}
\big), \notag
\end{align}
where $\overline X^i=X^i- \omega^i \omega_j X^j$ denotes the angular part
and $X^r= \omega_i X^i$ the radial part of $X$. Furthermore,
if $X$ has no angular part, that is $X=X^{u}(u,r)\del^B_{u}+X^{r}(u,r)\del^B_{r}$, then
the components of the tensor $\Acal$ in polar Bondi coordinates are
\begin{subequations}\label{A_coeff}
\begin{align}
\Acal^{uu}  & = 
-   \del_r^B X^u , 
\, \label{A_coeff1}\\
\Acal^{rr}  & =  
\frac{1}{2}\del_r^B X^r - \frac{1}{2}{\del_u^B} X^u-  {\del_u^B} X^r
- \Big(\frac{X^r}{r}-X \ln \Omega \Big), 
\label{A_coeff2}\\
\Acal^{ur}  & =   \frac{1}{2}\del_r^B X^u+ \Big(\frac{X^r}{r}-X \ln \Omega \Big),
\label{A_coeff3}\\
\Acal^{bc}  & = \Big(\frac{X^{r}}{r}
- \frac{1}{2}{\del_u^B} X^u- \frac{1}{2}\del_r^B X^r
- \Big(\frac{X^r}{r}-X \ln \Omega\Big) \Big)\frac{\delta^{bc}}{r^{2}}
.\label{A_coeff4}
\end{align}
\end{subequations}
\end{lemma}


\begin{lemma}[Identities for the flux terms]
\label{lem-post} 
Let $\Omega\geq 0$ be a weight and $X$ be a smooth vector field $X$ in Bondi coordinates. Then 
one has the following
identity for boundary terms on the equation \eqref{div_identity1_cut} (where $\Omega>0$)
\bel{form_bdry}
\aligned
{}^{(X)}\! \td{P}_\alpha N{'}^\alpha [\Phi] 
& = 
X^{u}(1-{r \over t} )(\del^B_{u}\Phi)^2+
\frac{1}{2}(X^{u}+X^{r}(1+ {r \over t} ))(\del^B_{r}\Phi)^{2} -X^{u}(1-{r \over t} )\del^B_u\Phi\del_r^B\Phi
\\
& \quad 
+ \frac{1}{2}(X^{u}+X^{r}(1-{r \over t} ))|\snabla\Phi|^{2} 
 +
 \frac{\Lambda}{\Omega^2} \Bcal\Phi^2. 
\endaligned
\ee
Assume that  $\Bcal \geq 0$ and $X$ is of the form
\be
X=X^{u}\del^B_{u}+X^{r}\del^B_{r} 
\quad {\text{with }} X^{u}(u,r)\geq 0, \  X^{r}(u,r) \geq 0. 
\label{positive_X}
\ee
Then, within the domain $\Kcal$, the flux term ${}^{(X)}\! \td{P}_\alpha N{'}^\alpha $ is equivalent
to a weighted sum of squares, namely 
\be
{}^{(X)}\! \td P_{\alpha} N{'}^\alpha 
\approx X^{u}(1-{r \over t} )
(\del^B_{u}\Phi)^2
+ \big(X^{u}+X^{r}(1+ {r \over t} )\big)(\del^B_{r}\Phi)^2
+ \big(X^{u}+X^{r}(1-{r \over t} )\big)|\snabla\Phi|^2 + \frac{\Lambda}{\Omega^2} \Bcal\Phi^2.
\label{coercive_bdry1}
\ee
\end{lemma}

 We now specialize our identities to the weights \eqref{weights_definition}.

\begin{proposition}[Vanishing of the conformal potentials]\label{vanish_V}
Recall that $V= \Omega^3 \Box(\Omega^{-1})$. 
For the conformal weights $\Omega_{r}=r$ and $\Omega_{s}=s^{2}$
and for all $r>0$ one has 
$
\Box(\Omega_{r}^{-1})  =   0$
and $ \Box(\Omega_{s}^{-1})  =   0$. 
In particular,  the corresponding potentials, say $V_{r}$ and $V_{s}$, vanish identically in \eqref{trans_eq}-- \eqref{AB_formulas1}, and on each hypersurface of constant $s$ the flux term in the conformal multiplier identity \eqref{div_identity1_cut} is positive definite.
\end{proposition}

As we will see later, for appropriate choices of vector field $X$,
the divergence term in the identity \eqref{div_identity1_cut} 
also has a good sign.

\begin{proof}
For the weight of $\Omega_{r}=r$ it suffices to check that
$V_{r}$ vanishes away from the singular set $r=0$.
A simple calculation using Bondi coordinates gives:
$
\Box(\Omega_{r}^{-1})  =
( (\del^B_{r})^2  + 2r^{-1} \del^B_{r}) (r^{-1}) = 0 
$
valid on the set $\Kcal\cap \{ r>0 \}$.
For the second weight, since $\Kcal\cap \{s^{2} =0\} = \emptyset$,
it suffices to check that $V_{s}=0$ throughout $\Kcal$.
A calculation gives, for all $(t,x)\in\Kcal$:
\begin{align*}
\Box(\Omega_{s}^{-1})  & = 
(-2\del^B_u \del^B_r 
+ (\del^B_{r})^2 - 2r^{-1}\del^B_u + 2r^{-1} \del^B_{r} )((u^2+2ru)^{-1})  
\\
& = (u^2+2ur)^{-3}(-16(u^2+ur)+4(u^2+2ur)+8u^2) + 4(u+2ur)^{-2} =  0.
\qedhere
\end{align*}
\end{proof}


\section{Conformal hyperboloidal energy estimates} 
\label{section5}

\subsection{Organization of this section}

Our main estimates for the wave equation are established in the present section. After preliminary material on functional or algebraic inequalities, we consider our three choices of conformal weight in \eqref{r_vf_choice1}, \eqref{r_vf_choice2}, and \eqref{s_vf_choice}, below, and we establish multiplier estimates in Proposition~\ref{proposition-555}.  
In turn, we are in a position to give a proof of Theorem~\ref{theor-main_estimate} about the fractional Morawetz energy estimate for the wave equation, as stated in the introduction. 

\subsection{Preliminary estimates}

We begin by restating a Hardy inequality established in \cite[Lemma 2.4]{PLF-YM-book}. 

\begin{lemma}[Hardy inequality] 
For all sufficiently regular and compactly supported functions $\phi$ defined in $\Kcal_{[s_0,s_1]}$ and for all $0<s_0 \leq s$ one has 
\be
\lp{r^{-1}\phi}{L^2(\Hcal_s)} 
\lesssim 
\sum_{i} \lp{\delu_{i}\phi}{L^2(\Hcal_s)}. 
\label{hardy1}
\ee
\end{lemma}

\begin{lemma}[Estimates for the weights]
\label{pos_weight1} 
The following positivity condition holds in the domain $\Kcal$ 
\bel{pos_weight2}
q(\underline{u},u)_{a}:=\Big(\frac{\baru^{2a}-u^{2a}}{r}-2a (\baru^{2a-1}+u^{2a-1})\Big) \geq  0, 
\qquad  a \in [1/2, 1], 
\ee
in which the expression vanishes identically at the end points $a=1$ and $a=1/2$. Additionally, one has 
\bel{pos_weight3}
(\baru^{2a}-u^{2a}) \approx a(rt^{2a-1}), 
\qquad  a \in [1/2, 1].
\ee
\end{lemma}

\begin{proof}
As in \cite{LindbladSterbenz}, we introduce $\underline{v}(\baru)= \baru^{2a}$,
$v(u)= u^{2a}$. For each fixed $t$, we introduce the function $f(r) = \underline{v}(t+r)-v(t- r)$.
This function satisfies $f(0)=0$ and $f'(r)= \dot{\underline{v}}+ \dot{v}$. The inequality
\eqref{pos_weight2} is then equivalent to
$f(0)-f(r)-f'(r)(0-r)\leq 0$.  
An immediate computation gives us 
\begin{align}
f'(r) \ & =  2a(t+r)^{2a-1}+2a(t- r)^{2a-1},  \qquad 
f''(r)  =  (2a-1)2a(\baru^{2a-2}-|u|^{2a-2}).
\label{pos_weight4}
\end{align}
By our conditions on $a$, we have $0\leq (2a-1)2a$ and $(2a-2)\leq 0$.
Therefore, $f''(r)\leq 0$ and the desired inequality is a convexity inequality.
Next, to derive the equation \eqref{pos_weight3} we rewrite the expression:
$\baru^{2a}-u^{2a} = t^{2a}\big((1+t^{-1}r)^{2a}-(1-t^{-1}r)^{2a} \big)$
and introduce the function $g(\omega):=(1+ \omega)^{2a}-(1- \omega)^{2a}$
with $\omega=t^{-1}r$.  It suffices now to show that there exist constants $C_0,C_1>0$ such that
$
C_0a\omega \leq g(\omega)
 \leq C_1a\omega, \label{pos_weight5}
$
with $0\leq \omega<1$ inside the domain $\Kcal$. Taking the derivative of this function gives us
\be
g'(\omega)=2a\big((1+ \omega)^{2a-1}+(1- \omega)^{2a-1} \big). \notag
\ee
For all $a\in[1/2,1]$, $g'(\omega)$ clearly satisfies the explicit bound
$2a \leq g'(\omega)
 \leq 2a(2^{2a-1}+1)$. 
Integrating these inequalities yields \eqref{pos_weight5}
with $C_0=2$ and $C_1=2(2^{2a-1}+1)$, which completes the proof of the lemma.
\end{proof}

\begin{remark}\label{distinct_multvf}
In the range $a\in[0,1/2)$, we have $q(\underline{u},u)_{a} \leq   0$,
and the expression vanishes identically for $a=0$.
Comparing this result with the inequality \eqref{pos_weight2}
we see why we need distinct vector field multipliers in order to cover both of the ranges
$a\in [0,1/2)$ and $a\in [1/2,1]$.
\end{remark}

The following Sobolev estimate is a consequence of the standard Sobolev estimate on hyperboloids. Observe that it allows any non-negative value of the exponent $a$. It is derived  from the standard Sobolev estimate on hyperboloids applied to $(s^a /t) \phi$ and we omit the details. 

\begin{lemma}
[Weighted Sobolev estimate on hyperboloids]
Let $s_1\geq s\geq s_0$ and $a \geq 0$. For any sufficiently regular function $\phi$ supported
in the interior of the future light cone $u=0$, one has 
\begin{equation}
\lp{s^a \tau_+^{1/2}\phi(s)}{L^\infty(\Hcal_s)}
\lesssim 
\sum_{|J|\leq 2} \lp{\frac{s^a }{t}L^J \phi(s)}{L^2 (\Hcal_s)}. \label{KlSob}
\end{equation}
\end{lemma}


\subsection{Conformal hyperboloidal energy multipliers}

We make now three choices of weights and vector fields and apply the conformal multiplier identity \eqref{div_identity1_cut}:
\begin{align}
K^a  & = (1+u^{2a})\del^B_{u}+ \frac{1}{2}(\baru^{2a}
-u^{2a})\del^B_{r} \, \qquad 
&&\Omega=r, \qquad
 \Lambda=1, \qquad 
 && 1/2\leq a\leq 1,   \label{r_vf_choice1}
\\
Y^a  & = (1+2a\tau_+^{2a}\tau_0)\del^B_{u}
+r^{2a}\del^B_{r}, \qquad &&\Omega=r, 
\qquad \Lambda=1, \qquad 
&& 0\leq a\leq 1/2, \label{r_vf_choice2}
\\
K & = (1+u^{2})\del^B_{u}
+2(u+r)r\del^B_{r}, \qquad &&\Omega=s^2, 
\qquad \Lambda=s^{2a-2}, \qquad 
&& 0\leq a\leq 1. \label{s_vf_choice}
\end{align}
This choice of vector fields can be motivated as follows: \eqref{r_vf_choice1}
smoothly interpolates between the scaling vector field ($a=1/2$) and the classical Morawetz vector field
$(a=1)$. The second choice  \eqref{r_vf_choice2} is analogous to the scaling vector field
with a scaling that is chosen to produce a good sign in the divergence terms (cf.~Remark
\ref{distinct_multvf}). Furthermore,
this vector field smoothly interpolates between the (Minkowski) Killling field
$\partial_t$ and the scaling vector field. Last, for the
third choice  \eqref{s_vf_choice} there is some flexibility: as long the vector field $\nabla^{\alpha}\Lambda$
is timelike and future directed we would have a good sign for the flux terms. However, the
decay statement needed in the present paper requires us to chose a weight $\Lambda$
of the form above.

\begin{proposition}[Main multiplier estimates]
\label{proposition-555} 
Let $\phi$ be a sufficiently regular and compactly supported solution to the equation \eqref{equa-Cauchy-wave-eq}. For any $s_0 \leq s \leq s_1$,
$ 0\leq a \leq 1$ the following estimates hold: 
\begin{align}
\sup_{s_0\leq s\leq s_1}\lp{\phi(s)}{\Ewave} 
  & \lesssim  \lp{\phi(s_0)}{\Ewave}+
\lp{F}{N_{W}^{0}[s_0,s_1]},
 \label{energy_est1}\\
\sup_{s_0\leq s\leq s_1}\lp{r\phi(s)}{\Eonea} 
  & \lesssim  \lp{r\phi(s_0)}{\Eonea} 
  + \lp{F}{N_{W}^a [s_0,s_1]},
 \label{est_weight1} \\
\sup_{s_0\leq s\leq s_1}\lp{s^2\phi(s)}{\Etwoa} 
  & \lesssim  \lp{s_{0}^2\phi(s_0)}{\Etwoa} 
  + \lp{F}{N_{
  W}^a [s_0,s_1]}.
 \label{est_weight2}
\end{align}
\end{proposition}

\begin{proof}[Proof of the estimate \eqref{energy_est1}]
Substituting the vector field $X= \del^B_u$ as well as $\Omega=1$
into the equation \eqref{divergence2}
and using the fact that $\mathcal{L}_{\del^B_u}\eta=0$
give us
\be
\int_{s=s_1} \!{}^{(X)}\!P_{\alpha}N^\alpha  \,
\Vol    =  \int_{s=s_0} \!{}^{(X)}\!P_{\alpha}N^\alpha  \,
\Vol  - \int_{s_0}^{s_1}\!\!\!\int_{s=s'}
 F  \del^B_u\phi  \  \Vol ds'.\label{energy_T1}
\ee

\vskip.3cm 

1. {\bf Boundary terms.} In view of the equation \eqref{rescaleddx}
it suffices to compute the integral of ${}^{(X)}\!P_{\alpha}N{'}^\alpha  dx$.
Using the equation \eqref{form_bdry}
gives us
\begin{align}
{}^{(X)}\!P_{\alpha} N{'}^\alpha    & = (1-{r \over t} )(\del^B_{u}\phi)^2+
\frac{1}{2}((\del^B_{r}\phi)^2+|\snabla\phi|^2)
-(1-{r \over t} )\del^B_u\phi\del_r^B\phi. \notag 
\end{align}
Applying the equation \eqref{coercive_bdry1} and integrating then yields (with $\td \phi=\phi$
in this case)
\begin{align}
\int_{s=s'} {}^{(X)}\!P_{\alpha} N^\alpha  \Vol 
= \int_{s=s'} {}^{(X)}\!P_{\alpha}N{'}^\alpha  dx
\approx
\int_{s=s'}(1-{r \over t} )
(\del^B_{u}\phi)^2
+(\del^B_{r}\phi)^2
+|\snabla\phi|^2  dx  
\approx \lp{ \phi(s)}{\Ewave}^2, 
\label{energy_T3}
\end{align}
where we used 
\be
1-{r \over t} = \Big(\frac{s}{t} \big)^2\frac{t}{t+r}
\approx \Big(\frac{s}{t} \big)^2
\quad 
\text{  in the domain $\Kcal_{[s_0,+\infty)}$.}
\label{weight_basic1}
\ee

\vskip.3cm 

2. {\bf Source term}. The term containing $F \del^B_u\phi$ has no sign
therefore we take absolute value and apply Cauchy-Schwarz to get
\begin{align}
\big|\int_{s_0}^{s_1}\!\!\!\int_{s=s'}
 F \del^B_u\phi  \  \Vol ds'\big|  & \leq  \int_{s_0}^{s_1}\!\!\!\int_{s=s'}
 \big|F \frac{s}{t}\del^B_{u}\phi \big|  \  dxds'  
 \leq   \lp{F}{N_{W}[s_0,s_1]}
   \sup_{s_0\leq s\leq s_1}\lp{\phi(s)}{\Ewave} 
 \label{energy_T4}.
\end{align}

\vskip.3cm 

3. {\bf Putting it all together.} Substituting \eqref{energy_T3} and \eqref{energy_T4}
into the equation \eqref{energy_T1} yields us 
$$
\lp{\phi(s_1)}{\Ewave}^{2} \lesssim  \lp{\phi(s_0)}{\Ewave}^{2}+
\lp{F}{N_{W}[s_0,s_1]}  
\sup_{s_0\leq s\leq s_1}\lp{\phi(s)}{\Ewave}.
$$
Taking the supremum and dividing by
$\sup\limits_{s_0\leq s\leq s_1}\lp{\phi(s)}{\Ewave}$
complete the proof of the estimate \eqref{energy_est1}.
\end{proof}

\begin{proof}[Proof of the estimate \eqref{est_weight1}
in the range $\frac{1}{2}\leq a \leq 1$]
We use the multiplier identity \eqref{div_identity1_cut}
with the choice  \eqref{r_vf_choice1}, that is, 
\[X=K^a  = (1+u^{2a})\del^B_{u}+ \frac{1}{2}(\baru^{2a}
-u^{2a})\del^B_{r} \, \qquad \Omega=r, \qquad
 \Lambda=1.  \]
Setting $\Phi=r\phi$ and using the formulas in line \eqref{AB_formulas1} we find 
\bel{energy_Khigh}
\aligned
&  \int_{ s=s_1} \!\!
{}^{(X )} \td{P}_\alpha[\Phi] N^\alpha  r^{-2} \
\Vol + \int_{s_0}^{s_1}\!\!\!\int_{s=s'} 
 r^{-2} \Acal^{\alpha\beta}\del_\alpha
\Phi\del_\beta\Phi  \  \Vol ds'   
\\
& = \int_{ s=s_0}\!\! {}^{(X )} \td{P}_\alpha
[\Phi]
N^\alpha  r^{-2} \ \Vol  
-  \int_{s_0}^{s_1}\!\!\!\int_{s=s'} 
F   r^{-1}X\Phi  \  \Vol ds',
\endaligned
\ee
where we used Lemma \ref{vanish_V}
to see that $\Bcal(X, r)\equiv0$.

\vskip.3cm 

1. {\bf Terms $\Acal^{\alpha\beta}$.} 
A simple computation gives us
\be
\frac{(K^a )^{r}}{r}-K^a (\ln r )=
\frac{(K^a )^{r}}{r}(1- \del^B_{r}r)= 0, 
\label{xr_term}
\ee
and substituting this \eqref{A_coeff} yields us 
\begin{subequations}
\begin{align}
\Acal^{uu}  & = 
-   \del_r^B(X^u)=
-   \del_r^B(u^{2a})= 0,
\qquad \qquad \qquad 
\Acal^{ur}  =  
 \frac{1}{2}\del_r^B X^u
= \frac{1}{2}  \del_r^B(u^{2a})= 0,
\, \notag \\ 
\Acal^{rr}  & =  
\frac{1}{2}\del_r^B X^r - \frac{1}{2}{\del_u^B} X^u-  {\del_u^B} X^r
 =  \frac{1}{4}\del^B_{r}(\baru^{2a}-u^{2a})
 - \frac{1}{2}\del^a _{u}(u^{2a})- \frac{1}{2}\del^B_{u}(\baru^{2a}
 -u^{2a}) \notag \\
 & = \frac{1}{4}(4a\baru^{2a-1})-
 \frac{1}{2}(2au^{2a-1})- \frac{1}{2}(2a\baru^{2a-1}
 -2au^{2a-1}) = 0, \notag \\
\Acal^{bc}  & = \big(\frac{X^{r}}{r}
- \frac{1}{2}{\del_u^B} X^u- \frac{1}{2}\del_r^B X^r
 \big)\frac{\delta^{bc}}{r^{2}}
= \big( \frac{1}{2r}(\baru^{2a}-u^{2a})
- \frac{1}{2}(2a\baru^{2a-1})- \frac{1}{4}
(4a\baru^{2a-1}) \big)\frac{\delta^{bc}}{r^{2}} \notag\\
& = \big( \frac{1}{2r}(\baru^{2a}-u^{2a})
-a(\baru^{2a-1}+u^{2a-1})\big)\frac{\delta^{bc}}{r^{2}}
=q(\underline{u},u)\frac{\delta^{bc}}{2r^{2}}
 \notag.
\end{align}
\end{subequations}
We then apply Lemma \ref{pos_weight1} to the weight in the last line above 
and conclude that, in the range $\frac{1}{2}\leq a \leq 1$, 
$\Acal^{bc}\geq 0$ and therefore
\bel{A_onehalf_pos}
A^{\alpha\beta}\del_\alpha
\Phi\del_\beta\Phi  \geq 0, 
\ee
where equality holds in the cases $a= 1/2$ and $a=1$.

\vskip.3cm

2. {\bf Boundary terms.} 
We compute ${}^{(X )} \td P_{\alpha} N{'}^\alpha $
using the equation \eqref{normalv} in Bondi coordinates:
\begin{align*}
{}^{(X )}\! \td{P}_\alpha N{'}^\alpha  & = 
(1+u^{2a})(1-{r \over t} )(\del^B_{u}\Phi)^2+
\frac{1}{2}((1+u^{2a})+ \frac{1}{2}(\baru^{2a}-u^{2a})(1+ {r \over t} ))(\del^B_{r}\Phi)^{2}\\
& \quad  + \frac{1}{2}((1+u^{2a})+ \frac{1}{2}(\baru^{2a}-u^{2a})
(1-{r \over t} ))|\snabla\Phi|^{2} 
 -(1+u^{2a})(1-{r \over t} )\del^B_u\Phi\del_r^B\Phi.
\end{align*}
Multiplying by $r^{-2}$, integrating,
and applying the equation \eqref{coercive_bdry1}, we obtain 
\begin{align}
& \int_{s=s'} {}^{(X)}\!P_{\alpha}N^\alpha  \, r^{-2}\Vol 
= \int_{s=s'} {}^{(X)}\!P_{\alpha}N{'}^\alpha  \, r^{-2}dx  \label{boundary_Ka1} \\
& \approx
\int_{s=s'}\Big( 
(1+u^{2a})(1-{r \over t} )(\del^B_{u}\Phi)^2+
((1+u^{2a})+(\baru^{2a}-u^{2a})(1+ {r \over t} ))(\del^B_{r}\Phi)^{2}
\notag \\
& \qquad\qquad  \quad +((1+u^{2a})+(\baru^{2a}-u^{2a})
(1-{r \over t} ))|\snabla\Phi|^{2}  \Big)
 \, r^{-2}dx  \notag \\
& \approx \lp{r\phi(s)}{\Eonea}^2.
\notag 
\end{align}
Here, we used the equations \eqref{pos_weight3} and \eqref{weight_basic1} in order 
to control the weights on the last line.

\vskip.3cm

3. {\bf Source term.} We take absolute value and apply Cauchy-Schwarz to get
\begin{align}
\big|\int_{s_0}^{s_1}\!\!\!\int_{s=s'}
F   \frac{K^a \Phi}{r}   \  \Vol ds'\big|  & \lesssim  \int_{s_0}^{s_1}\!\!\!\int_{s=s'}
 \big| F (1+ u^{2a})\frac{s}{t}\frac{\del^B_{u}\Phi}{r} \big|  
 + \big| F \tau_+^{2a}\frac{s}{t}\frac{\del^B_{r}\Phi}{r} \big|  \  dxds' 
  \label{energy_h1_mult} \\ 
& \lesssim \lp{F}{N_{W}^a [s_0,s_1]} 
\sup_{s_0\leq s\leq s_1}\lp{r\phi(s)}{\Eonea}  \notag.
\end{align}

\vskip.3cm

4. {\bf Putting it all together.} Substituting the results of Eqs.
\eqref{A_onehalf_pos}, \eqref{boundary_Ka1},
and \eqref{energy_h1_mult}
into the equation \eqref{energy_Khigh} yields:
$$
\lp{r\phi(s_1)}{\Eonea}^2  \lesssim 
\lp{r\phi(s_0)}{\Eonea}^2+
\lp{F}{N_{W}^a [s_0,s_1]} 
\sup_{s_0\leq s\leq s_1}\lp{r \phi(s)}{\Eonea}. 
$$
Taking the supremum and dividing by
$\sup\limits_{s_0\leq s\leq s_1}\lp{r\phi(s)}{\Eonea}$
complete the proof of \eqref{est_weight1}.
\end{proof}

\begin{proof}[Proof of the estimate \eqref{est_weight1}
for the range $ 0 < a \leq 1/2$]
We now use the multiplier identity the equation \eqref{div_identity1_cut}
with the choice \eqref{r_vf_choice2}, that is, 
\[X=Y^a  = (1+2a\tau_+^{2a}\tau_0)\del^B_{u}
+r^{2a}\del^B_{r}, \qquad \Omega=r ,
\qquad \Lambda=1.\]
Setting $\Phi=r\phi$ and using the formulas in line \eqref{AB_formulas1}, once again we get
$$
\aligned
&  \int_{ s=s_1} \!\!
{}^{(X )} \td{P}_\alpha[\Phi] N^\alpha  r^{-2} \
\Vol + \int_{s_0}^{s_1}\!\!\!\int_{s=s'} 
 r^{-2} A^{\alpha\beta}\del_\alpha
\Phi\del_\beta\Phi  \  \Vol ds'    \label{energy_Klow} 
\\
& = \int_{ s=s_0}\!\! {}^{(X )} \td{P}_\alpha
[\Phi]
N^\alpha  r^{-2} \ \Vol  
-  \int_{s_0}^{s_1}\!\!\!\int_{s=s'} 
F   r^{-1}X \Phi  \  \Vol ds',
\endaligned
$$
where we used Lemma \ref{vanish_V} in order to conclude that $\Bcal(X, r)\equiv0$.

\vskip.3cm

1. {\bf  Terms $\Acal^{\alpha\beta}$.} Since $\Omega_{r}= r$
we apply the equation \eqref{xr_term}, use \eqref{A_coeff}, and find 
\begin{subequations}
\begin{align}
\Acal^{uu}  & = 
-   \del_r^B(X^u)=
-   \del_r^B(2a\tau_+^{2a}\tau_0)= 2a(1-2a) 
\tau^{2a-1}_+ \tau_0, \notag \\
\Acal^{ur}  & =  
 \frac{1}{2}\del_r^B X^u
= \frac{1}{2}   \del_r^B(2a\tau_+^{2a}\tau_0)= 
- a(1-2a) 
\tau^{2a-1}_+ \tau_0
 \notag, 
\end{align}
and 
\begin{align}
\Acal^{rr}  & =  
\frac{1}{2}\del_r^B X^r - \frac{1}{2}{\del_u^B} X^u-  {\del_u^B} X^r
 =   ar^{2a-1}-a\tau_+^{2a-1}\frac{u}{\tau_-}-a(2a-1)
 \tau_+^{2a-1}\tau_0\notag \\
 & = a(1-2a)\tau_+^{2a-2}\underline{u}\tau_0 
 + a(r^{2a-1}- \tau_+^{2a-1}\frac{u}{\tau_-})
  \geq  a(1-2a)\tau_+^{2a-1} \tau_0 
,  \notag \\
\Acal^{bc}  & = \big(\frac{X^{r}}{r}
- \frac{1}{2}{\del_u^B} X^u- \frac{1}{2}\del_r^B X^r
 \big)\frac{\delta^{bc}}{r^{2}}
= \big( r^{2a-1}
-a\tau_+^{2a-1}\frac{u}{\tau_-}-a(2a-1)
 \tau_+^{2a-1}\tau_0-
ar^{2a-1} \big)\frac{\delta^{bc}}{r^{2}} \notag\\
& = \big( (1-a)r^{2a-1}-a\tau_+^{2a-1}\frac{u}{\tau_-}
+ a(1-2a)
 \tau_+^{2a-1}\tau_0\big)\frac{\delta^{bc}}{r^{2}}\geq 0
 \notag, 
\end{align}
\end{subequations}
where we used that
$r^{2a-1} \geq \tau_+^{2a-1}$ in the range $0<a\leq \frac{1}{2}$.
Combining these results gives us
\be
\Acal^{uu}  \geq  0, 
\qquad
 \Acal^{rr}  \geq     0,  
 \qquad 
|\Acal^{ur}|  \leq \sqrt{\Acal^{uu} \Acal^{rr}},
\qquad
 \Acal^{bc} \geq  0.
\ee
An application of the classical inequality of
arithmetic and geometric means then yields the
favorable sign
\be
A^{\alpha\beta}\del_\alpha
\Phi\del_\beta\Phi  \geq 0.
\label{A_onehalf_low_pos}
\ee

\vskip.3cm

2. {\bf Boundary terms.} 
We compute ${}^{(X )}\!P_{\alpha}N^\alpha $
using the equation \eqref{normalv} in Bondi coordinates 
\begin{multline}
{}^{(X)}\! \td{P}_\alpha N{'}^\alpha  
= 
(1+2a\tau^{2a}_+ \tau_0)(1-{r \over t} )(\del^B_{u}\Phi)^2+
\frac{1}{2}(1+2a\tau^{2a}_+ \tau_0+r^{2a}(1+ {r \over t} ))(\del^B_{r}\Phi)^{2}\\
+ \frac{1}{2}(1+2a\tau^{2a}_+ \tau_0+r^{2a}(1-{r \over t} ))|\snabla\Phi|^{2} 
 -(1+2a\tau^{2a}_+ \tau_0)(1-{r \over t} )\del^B_u\Phi\del_r^B\Phi.\notag
\end{multline}
Multiplying by $r^{-2}$, integrating,
and applying the equation \eqref{coercive_bdry1} yields
\begin{align}
& \int_{s=s'} {}^{(X)}\!P_{\alpha}N^\alpha  \, r^{-2}\Vol  
= \int_{s=s'} {}^{(X)}\!P_{\alpha}N{'}^\alpha  \, r^{-2}dx\label{boundary_Klow} \\
& \approx
\int_{s=s'}\Big(
 1+2a\tau^{2a}_+ \tau_0(1-{r \over t} )(\del^B_{u}\Phi)^2
+(1+2a\tau^{2a}_+ \tau_0+r^{2a}(1+ {r \over t} ))(\del^B_{r}\Phi)^{2} \notag \\
& \hspace{1in}  +(1+2a\tau^{2a}_+ \tau_0+r^{2a}(1-{r \over t} ))|\snabla\Phi|^{2} 
 \Big) \, r^{-2}dx  \notag 
 \hskip2.cm 
  \approx   \lp{r\phi(s)}{\Eonea}^2, 
\notag 
\end{align}
where we have used  \eqref{weight_basic1} in order to bound the weights on the last line.

\vskip.3cm

3. {\bf Source term.} Similarly to what was done before, we have now
\begin{align}
\big|\int_{s_0}^{s_1}\!\!\!\int_{s=s'}
F   \frac{X \Phi}{r}   \  \Vol ds'\big|  & \lesssim  \int_{s_0}^{s_1}\!\!\!\int_{s=s'}
 \big| F \tau^{2a}_+ \tau_0\frac{s}{t}\frac{\del^B_{u}\Phi}{r}  \big|  
 + \big| F \tau_+^{2a}\frac{s}{t}\frac{\del^B_{r}\Phi}{r} \big|  \  dxds' 
  \label{energy_l1_mult} \\ 
& \lesssim
 \lp{F}{N_{W}^a [s_0,s_1]} 
\sup_{s_0\leq s\leq s_1}\lp{r\phi(s)}{\Eonea}^2
\notag.
\end{align}

\vskip.3cm

4. {\bf Putting it all together.} The conclusion follows
from substituting the estimates \eqref{A_onehalf_low_pos}, \eqref{boundary_Klow},
and \eqref{energy_l1_mult} into \eqref{energy_Klow}
and simplifying the result.
\end{proof}

\begin{proof}[Proof of the estimate \eqref{est_weight2} for $ 0 <  a \leq 1$]
We now use the multiplier identity \eqref{div_identity1_cut}
with the choice \eqref{s_vf_choice}, that is, 
\[X=K = (1+u^{2})\del^B_{u}
+2(u+r)r\del^B_{r}, \qquad \Omega=s^2, 
\qquad \Lambda=s^{2a-2} . \]
Setting$\Phi=s^2\phi$ and using the formulas in line \eqref{AB_formulas1}, we obtain
\be
\aligned
& \int_{ s=s_1} \!\!
{}^{(X)} \td{P}_\alpha[\Phi] N^\alpha  s^{-4} \
\Vol    + \int_{s_0}^{s_1}\!\!\!\int_{s=s'} 
 \big(s^{2a-2}\Acal^{\alpha\beta}\del_\alpha
\Phi\del_\beta\Phi + 
s^{2-2a}\Ccal^{\alpha}{}^{(X)}\td{P}_{\alpha} \big)
 s^{-4} \  \Vol ds'  \label{energy_Kweight2}\\
& = \int_{ s=s_0}\!\! {}^{(X)} \td{P}_\alpha
[\Phi]
N^\alpha  s^{-4} \ \Vol  
-  \int_{s_0}^{s_1}\!\!\!\int_{s=s'} 
s^{2a-2}F   s^{-2}X\Phi  \  \Vol ds',
\endaligned
\ee
where we used Lemma \ref{vanish_V}
to see that $\Bcal(X, s^2)\equiv0$.

\vskip.3cm

1. {\bf  Terms $\Acal^{\alpha\beta}$.}
Since $\Omega= s^2$ we have
\begin{align}
\frac{K^{r}}{r}-K(\ln s^2 )& =
2(u+r)- \frac{1}{u\baru}(2(u+r)u^2+4(u+r)ur)
=0.\label{xs_term}
\end{align}
Combining this with the list of identities underneath the equation \eqref{xr_term} with $a=1$
we get
\be
\Acal^{uu}  =  0
\qquad
\Acal^{rr}  =   0,
\qquad
\Acal^{ur}  =  0,
\qquad
 \Acal^{bc} = 
\big( \frac{1}{2r}(\baru^{2}-u^{2})
-(\baru+u)\big)\frac{\delta^{bc}}{r^{2}} =  0, 
\ee
therefore 
\be
\label{A_weight2_sign}
A^{\alpha\beta}\del_\alpha
\Phi\del_\beta\Phi = 0.
\ee

\vskip.3cm

2. {\bf Term $\Ccal^{\alpha}$.}
Expanding
$
\Ccal^{\alpha} = \nabla^{\alpha}s^{2a-2}
=(2a-2)s^{2a-3}
\nabla^{\alpha}s, 
$
and since the case $a=1$ is trivial we may assume $0<a<1$.
We observe that the vector field
$- \nabla s=
- \eta^{\alpha\beta}\del_{\alpha}s\del_{\beta}
= \frac{t}{s}\del_t+ \frac{r}{s}\del_r
$
is future directed and causal.
In addition, the Morawetz vector field
$K$ is future directed and causal 
throughout the domain $\Kcal_{[s_0,+\infty)}$. Since
the classical energy momentum tensor $Q_{\alpha\beta}$
satisfies the dominant energy condition,  the sign condition
\be
 \label{error_weight2_sign}
\Ccal^{\alpha}{}^{(X)}\td{P}_{\alpha}s^{2-2a} =
\Ccal^{\alpha}{}^{(X)}{P}_{\alpha} =
(2-2a)s^{2a-3}
Q(- \nabla s, X) \geq 0
\ee
is
valid throughout the domain $\Kcal$.

\vskip.3cm

3. {\bf  Boundary terms.} 
We compute ${}^{(X)} \td P_{\alpha} N{'}^\alpha $
using the equation \eqref{normalv}.
Multiplying by $s^{-4}$, integrating,
and applying the equation \eqref{coercive_bdry1}, we find 
\be
\label{boundary_Kweight2}
\aligned
& \int_{s=s'} {}^{(X)}\!\td P_{\alpha}N^\alpha  \, s^{-4}\Vol  
= \int_{s=s'} {}^{(X)}\!\td P_{\alpha}N{'}^\alpha  \, s^{-4}dx 
 \\
& \approx
\int_{s=s'}s^{2a-2}_+ \Big(
 (1+u^{2})(1-{r \over t} )(\del^B_{u}\Phi)^2+
((1+u^{2})+(\baru^{2}-u^{2})(1+ {r \over t} ))(\del^B_{r}\Phi)^{2}\notag \\
& \qquad\qquad\qquad\qquad \qquad  +((1+u^{2})+(\baru^{2}-u^{2})
(1-{r \over t} ))|\snabla\Phi|^{2}  \Big) \, s^{-4}dx  
   \approx \lp{s^2\phi(s)}{\Etwoa}^2,
\endaligned
\ee
where we have used equation \eqref{pos_weight3} and
equation \eqref{weight_basic1}
to bound the weights on the last line.

\vskip.3cm

4. {\bf Source term.} Taking the absolute value and applying the Cauchy-Schwarz inequality, we get
\be\label{energy_h2_mult}
\aligned
& \big|\int_{s_0}^{s_1}\!\!\!\int_{s=s'}
F   s^{2a-2}_+ \frac{X\Phi}{s^2}   \  \Vol ds'\big|  
 \lesssim  \int_{s_0}^{s_1}\!\!\!\int_{s=s'}
 \big| F s^{2a-2}\tau^2_{-}\frac{s}{t}\frac{\del^B_{u}\Phi}{s^2} \big|  
 + \big| F s^{2a-2}\tau^2_+ \frac{s}{t}\frac{\del^B_{r}\Phi}{s^2} \big|  \  dxds' 
   \\ 
& \lesssim \lp{F}{N^a [s_0,s_1]} 
\sup_{s_0\leq s\leq s_1}\lp{s^2 \phi(s)}{\Etwoa}, 
\endaligned
\ee
where, for each of the weights, we used 
\[s^{2a-2}\tau^2_{-}\lesssim s^a (s^{a-1}\tau_-)
\lesssim  \tau_+^a (s^{a-1}\tau_-),
\qquad s^{2a-2}\tau^2_+ \frac{s}{t}\lesssim s^a(s^{a-1}\tau_+)
\lesssim \tau_+^a (s^{a-1}\tau_+).
 \]

\vskip.3cm

5. {\bf Putting it all together.} The desired concusion now follows
by substituting the estimates \eqref{A_weight2_sign},
\eqref{error_weight2_sign}, \eqref{boundary_Kweight2},
and \eqref{energy_h2_mult} into \eqref{energy_Kweight2}
and simplifying the result.
\end{proof}

\begin{lemma}[Combining the energies]
For all $s_0\leq s\leq s_1$, $ 0 < a \leq 1$
with $\Kcal_{[s_0,s_1]} \subset\Kcal$.
the lower bound holds: 
\begin{align}
\lp{\phi(s)}{E_{W}^{a} }  \lesssim   
 \lp{r\phi(s)}{\Eonea}+ \lp{s^2 \phi(s)}{\Etwoa}
+ \lp{\phi(s)}{E_{W}}\label{conjlemma}.
\end{align}
\end{lemma}

\begin{proof}
By recalling that $\del_t= \del^B_u$ and $\del^B_r= \del_t+ \del_r$, a simple computation yields the identities
\begin{align*}
& \frac{1}{r}\del^B_r(r\phi)
= \del^B_r\phi+ \frac{ \phi}{r},
&& \frac{1}{r}\del_t (r\phi)
= \del_t\phi , \\
& \frac{1}{s^2}\del^B_r(s^2\phi)= \del^B_r\phi
+ \frac{2}{\underline u}\phi,
&& \frac{1}{s^2}\del_t (s^2\phi)
= \del_t \phi+ \frac{2t}{\underline u u}\phi.
\end{align*}
We multiply the two identities in the right-hand side above by the weight $(s^a u/t)$
and subtract, namely 
\begin{align*}
s^{a-1}\frac{s}{t}u\Big(\frac{1}{s^2}\del_t (s^2\phi)- \frac{1}{r}\del_t (r\phi)\Big)
= \frac{s^a }{t} \, \frac{2t}{\underline u }\phi.
\end{align*}
Squaring both sides, expanding, using Young inequality and the fact that $1\leq 2t/\underline u\leq 2$
gives us
\be
\big|\frac{s^a  \phi}{t}\big|^2\lesssim
s^{2a-2}(1+u^2)\big|
\frac{s}{t}\cdot \frac{1}{s^2}\del_t(s^2\phi)\big|^2
+ \tau^{2a}_+ \tau^{\max\{2a,1\}}_{0}\big|\frac{s}{t} \, \frac{1}{r}\del_t(r\phi)\big|^{2}
  \label{conjlemma2}, 
\ee
where on the weight for  the last term in the right-hand side  we used the following  inequalities valid in the domain $\Kcal_{[s_0,+\infty)}$:
\begin{align*}
& \frac{s^a u}{t} =s^{a-1}u\frac{s}{t}
= u^a (\frac{u}{s})^{1-a}\frac{s}{t}\leq u^a \frac{s}{t},  
&& \frac{1}{2}\leq a \leq 1,
\\
& \frac{s^a u}{t}
\leq \tau_+^a \frac{u}{s}\frac{s}{t}
= \tau_+^a \big(\frac{u}{\underline u} \big)^{1/2}\frac{s}{t}
\lesssim \tau_+^a \tau_0^{1/2}\frac{s}{t},  
&&0< a< \frac{1}{2}. 
\end{align*}
Thus we have the lower bound
\begin{align}
\tau^{2a}_+ \tau^{\max\{2a,1\}}_{0}\big| \frac{s}{t}\del_t \phi\big|^2+ \big|\frac{s^a  \phi}{t}\big|^2
\lesssim s^{2a-2}\tau_-^2\big| \frac{s}{t} \frac{1}{s^2}\del_t(s^2\phi)\big|^2
+ \tau^{2a}_+ \tau^{\max\{2a,1\}}_{0}\big| \frac{s}{t} \frac{1}{r}\del_t(r\phi)\big|^2.
\label{conjlemma1}
\end{align}
Next we derive a bound for tangential derivatives to the hyperboloids $\underline\nabla_x$.
Observe that
\[\del^B_r= \del_t+ \del_r= \big(1-{r \over t} \big)\del_t+ \big({r \over t} \del_t+ \del_r\big)
= \big(\frac{u}{t} \big)\del_t+ \underline \del_r.\]
Multiplying by $s^{a-1}\underline u(s/t)$, conjugating by the weight $\Omega_{s}=s^2$
and rearranging yields
\be
s^a\frac{\underline u}{t}\frac{\underline \del_r(s^2\phi)}{s^2}=
s^{a-1}\underline u\frac{s}{t} \big(\frac{\del^B_r(s^2\phi)}{s^2} \big) -
\frac{s^{a-1} u\underline u}{t}\big(\frac{s}{t}\frac{1}{s^2}\del_t (s^2\phi)\big).
\label{conjlemma3}
\ee
Squaring both sides, using Young inequality and the fact that  $\underline \del_r(s)=0$
together with the bound $1\leq \underline u/t\leq 2$ gives the estimate
\begin{align}
s^{2a}\big|\underline \del_r\phi\big|^2
\lesssim s^{2a-2}\tau_-^2\big| \frac{s}{t}\cdot \frac{1}{s^2}\del_t(s^2\phi)\big|^2+
 s^{2a-2}\tau_+^2\big|\frac{\del^B_r(s^2\phi)}{s^2}\big|^2. \label{conjlemma4}
\end{align}

To add the angular
derivatives we observe that, for any function on a fixed hyperboloid $s=const$, we have
$\del_i\phi(s)= \underline\del_i\phi(s)$. Therefore, after integrating on hyperboloids we can identify
$\underline\snabla\phi(s)= \snabla\phi(s)$. Consequently, adding \eqref{conjlemma1}
 and \eqref{conjlemma4}, integrating, and using this observation completes the
proof of inequality \eqref{conjlemma}.
\end{proof}

\begin{remark}
There is a small penalty for removing the conjugation and
going back to the semi-hyperboloidal frame. This is to be expected since we require
a weight of $(s/t)$ to use the $\del_t$ derivative in $L^{2}(\mathcal{H}_{s})$.
\end{remark}


\section{Boundeness property for the wave-Klein-Gordon model} 
\label{section6}

\subsection{Main linear estimates}

We are now in a position to combine our results so far and establish \eqref{frac_conf_wave1}, \eqref{frac_conf_wave2},
and \eqref{point1}.
For each of these inequalities, we first prove the case $k=0$
then add regularity by commuting the equation with the 
vector fields in the Lie algebra $\Lbb$. Using
the commutation relations
\begin{align}
&\left[\Box, \del^{I_0}_0\del^{I_1}_1\del^{I_2}_2\del^{I_3}_3 \right]=0,
&& I_0+ I_1+ I_2+ I_3 = |I|, 
\label{comm_trivial1}
\\
&\big[\Box, L^{J_1}_1 L^{J_2}_{2} L^{J_3}_{3 }\big]=0,
 &&J_1+ J_2+ J_3 = |J|,\label{comm_trivial2}
\end{align}
where $|I|+|J|=k$ then yields each corresponding result for $k=0,1, \ldots$

\begin{proof}[Proof of Theorem \ref{theor-main_estimate}]
The estimate \eqref{frac_conf_wave1} follows by adding
estimates \eqref{energy_est1}+\eqref{est_weight1}+\eqref{est_weight2}. 
The estimate \eqref{frac_conf_wave1} for each $k=0,1, \ldots$ follows, via induction, by commuting
the equation with products of vector fields in $\Lbb$
and applying the relations \eqref{comm_trivial1} and \eqref{comm_trivial2}.
\end{proof}

\begin{proof}[Proof of Corollary \ref{theor-main_cor}]

Estimate \eqref{frac_conf_wave2} in the case $k=0$ follows
by applying estimate \eqref{conjlemma} to the left-hand side of 
\eqref{frac_conf_wave1} with
$k=0$. The estimate \eqref{frac_conf_wave2}
for each $k=0,1, \ldots$ then follows, via induction, by commuting
the equation with products of vector fields in $\Lbb$ and using the relations
\eqref{comm_trivial1} and \eqref{comm_trivial2}.

On the other hand, the estimate \eqref{point1} in the case $k=0$ follows
by using estimate \eqref{KlSob} followed by
an application of the estimate \eqref{frac_conf_wave2}
with $k=2$. The estimate \eqref{point1}
for each $k=0,1, \ldots$ then follows, via induction, by commuting
the equation with products of vector fields in $\Lbb$ and using the relations
\eqref{comm_trivial1} and \eqref{comm_trivial2}.
\end{proof}


\subsection{Closing the bootstrap argument} 

We consider the local-in-time solution $(u,v)$ to the Cauchy problem associated with the system \eqref{eq main}
and, as usual, we proceed by bootstrap and assume that this solution is defined up to some hyperboloidal time 
$s_1 > s_0$ which is chosen to be maximal for the bounds under consideration. That is, the following energy bounds\footnote{In the high-order energy bound for the wave equation, the exponent $\max(1,k)) \delta$ provides a small growth of order $\delta$ for {\sl all} $k$, and we refer the reader to Appendix A in \cite{PLF-YM-1}.}
 are assumed on the interval $[s_0, s_1]$ (for some constants $C_1, \eps, \delta>0$, yet to be determined):  
\bel{ineq energy assumption}
\aligned
& \Ewavea(s,\del^IL^J u)^{1/2}\leq C_1\eps s^{(\max(a-1/2,0)  + \max(1,k)) \delta}, 
&&|J|=k, \quad &&&|I|+|J|\leq N, \qquad 
\\
& \Ewavea(s,\del^IL^J u)^{1/2}\leq C_1\eps s^{\max(a-1/2,0) \delta}, 
&& \quad &&&|I|+|J|\leq N-4,  
\\
& \EKG(s,\del^IL^J v)^{1/2}\leq C_1\eps s^{(\max(-a+1/2,0) + k)\delta}, 
&&|J|=k, \quad &&&|I|+|J|\leq N,  
\\
& \EKG(s,\del^IL^J v)^{1/2}\leq C_1\eps s^{k\delta}, 
&&|J|=k, \quad &&&|I|+|J|\leq N-4.  
\endaligned
\ee
We will establish improved bounds, with $C_1$ replaced by $C_1/2$, on the same interval of time (provided $\eps$ is sufficiently small and the constants are suitably chosen).

The existence theory for the model \eqref{eq main} was already established in \cite{PLF-YM-1} and what we need to check here is that the additional decay in space (assumed on the initial data) is propagated in time. The proof in \cite{PLF-YM-1} requires some adaptation which is now discussed. It is unnecessary to rewrite the proof completely, and we now explain the additional arguments (for any given $0\leq a \leq 1$).  
\bei

\item {\bf Commutator in the Klein-Gordon equation.} Let us consider the terms $u \del_\alpha\del_\beta v$ arising in the Klein-Gordon equation. We now control the wave component in a suitably weighted $L^2$ norm, 
namely the norm $\lp{\frac{s^a}{\tau_+}u}{L^2(\Hcal_s)}$.
Hence, the higher-order commutator terms for the Klein-Gordon equation can now be estimated as follows, in the case $N_1 \geq N_2$ (even when $N_1=N$ the maximum number of vector fields) and 

For any combination of boosts and translations $Z^{N} =\partial^{I}L^{J}$ with $|I + |J| \leq N$, we consider the quadratic expression 
$Z^{N_1} u \del^{2}Z^{N_2}v$ 
\be
\aligned
\lp{Z^{N_1} u \del^{2}Z^{N_2}v}{L^2(\Hcal_s)}
& \lesssim  
\lp{\frac{s^a}{\tau_+} Z^{N_1}u}{L^2(\Hcal_s)}
\, 
\lp{\frac{\tau_+}{s^a}\del^2Z^{N_2}v}{L^\infty(\Hcal_s)} 
\\
& \lesssim  
\lp{\frac{s^a}{\tau_+} Z^{N_1}u}{L^2(\Hcal_s)}
\,  \frac{\tau_+}{s^{a+1/2}t}
 = \lp{\frac{s^a}{\tau_+} Z^{N_1}u}{L^2(\Hcal_s)}
   \, \frac{1}{s^{a+1/2}}, 
\endaligned
\ee
where, in the second inequality, we have used the $L^\infty$ decay enjoyed by the Klein-Gordon component. 

\item {\bf First-order quadratic nonlinearities in the wave equation.} Let us consider next the terms $\del_\alpha v\del_\beta v$ arising in the wave equation, and recall that our norm involves the weight $s^{a-1}\tau_+$. 
Independently from the number of vector fields applied to each term, one of the two copies of the Klein-Gordon component is always estimated in $L^\infty$, hence we write for general exponents $N_1, N_2$ 
\begin{align*}
& \lp{s^{a-1}\tau_+ Z^{N_1}\del v Z^{N_2}\del v}{L^2(\Hcal_s)}
\lesssim  
\lp{\del Z^{N_1} v}{L^2(\Hcal_s)}
\lp{s^{a-1}\tau_+ \del Z^{N_2}v}{L^\infty(\Hcal_s)}
\\
&  \lesssim  
\lp{\del Z^{N_1} v}{L^2(\Hcal_s)}\frac{s^{a-1}\tau_+}{s^{1/2}t}
   = \lp{\del Z^{N_1} v}{L^2(\Hcal_s)}\frac{1}{s^{3/2-a}}. 
\end{align*}

\item {\bf Quadratic nonlinearity in the wave equation.} 
Similarly, for the wave equation again we write 
\be
\aligned
& \lp{s^{a-1}\tau_+ Z^{N_1} v Z^{N} v}{L^2(\Hcal_s)}
\lesssim  
\lp{ Z^{N_1} v} {L^2(\Hcal_s)} 
\lp{s^{a-1}\tau_+ Z^{N}v}{L^\infty(\Hcal_s)}
\\
& \lesssim  \lp{Z^{M} v}{L^2(\Hcal_s)}\frac{s^{a-1}\tau_+}{s^{1/2}t}
   = \lp{Z^{M} v}{L^2(\Hcal_s)}\frac{1}{s^{3/2-a}}. 
\endaligned
\ee
\eei 
Hence, for all nonlinearities of interest we can control the source-terms of both the wave and the Klein-Gordon equations and the remaining arguments  in \cite{PLF-YM-1} applies.  


\small
 
\paragraph*{Acknowledgments.} 
The first author (PLF) was partially supported by the ERC Innovative Training Network grant 642768. Part of this work was done when the first author was visiting the School of Mathematical Sciences at Fudan University, Shanghai, and the Courant Institute of Mathematical Sciences, New York University. The first author (PLF) is very grateful to the third author (YT) for the opportunity of visiting Kyoto University (where this paper was completed) as an international research fellow of the Japan Society for the Promotion of Science (JSPS). The second author (JO) gratefully acknowledges support from Institut Henri Poincar\'e, Paris.  



\small

\appendix 

\section{Proof of technical results}
\label{Append-A}

\begin{proof}[Proof of Lemma~\ref{moved-to-Append-A1}]
To establish the equation \eqref{bondi_deften_A1} we expand $X^{i} = \overline X^i+ \omega^i \omega_j X^j$
\be
\del^B_\gamma X^\gamma  = \del^B_u X^u + 
 \del^B_i(\omega^{i})\omega_{j} X^j
 + \omega^{i}\del^B_i(\omega^{i}X^j)
 + \del^B_i\overline{X}^i = 
\del^B_u X^u +  \frac{2X^{r}}{r}+ \del^B_r X^r+ \del^B_i\overline{X}^i 
\notag.
\ee
Using this together with the basic identity
${}^{(X)}\widehat{\pi}^{\alpha\beta}  =  - \mathcal{L}_{X}
\eta^{-1}- \eta^{\alpha\beta }\del_\gamma X^\gamma$,  
in \eqref{pi_hat_formula} gives the result.
To calculate all the components of $\Acal$ we use
Eq. \eqref{bondi_deften_A1} together with the equation \eqref{pi_hat_formula} to get
$$
\Acal^{\alpha\beta}  =   \frac{1}{2}\Big(
\del_\gamma^B(X^\alpha) \eta^{\gamma\beta}
+ \del_\gamma^B(X^\beta) \eta^{\alpha\gamma}
-X( \eta^{\alpha\beta})
-({\del_u^B} X^u + \del_r^B X^r  )\eta^{\alpha\beta}
\Big) - \Big( \frac{X^r}{r}-X \ln \Omega) \Big)\eta^{\alpha\beta}.
$$
We now compute all the components of this tensor. For $\Acal^{uu}$ we use the fact that
$\eta^{u\alpha} =-1$ if $\alpha=r$, and $\eta^{u\alpha} =0$ otherwise. This gives
\begin{align*}
\Acal^{uu}  & =   \frac{1}{2}\Big(
\del_\gamma^B(X^u) \eta^{\gamma u}
+ \del_\gamma^B(X^u) \eta^{\gamma u}
-X( \eta^{uu})
-({\del_u^B} X^u + \del_r^B X^r  )\eta^{uu}
\Big) - \Big( \frac{X^r}{r}-X \ln \Omega \Big) )\eta^{uu} 
 =   -   \del_r^B X^u.
\end{align*}
For $\Acal^{rr}$ we use the fact that
$\eta^{ur} =-1$, $\eta^{rr} =1$, and $\eta^{r\alpha} =0$ otherwise. This yields
\begin{align*}
\Acal^{rr}  & =   \frac{1}{2}\big(
\del_\gamma^B(X^r) \eta^{\gamma r}
+ \del_\gamma^B(X^r) \eta^{\gamma r}
-X( \eta^{rr})
-({\del_u^B} X^u + \del_r^B X^r  )\eta^{rr}
\big) - \Big( \frac{X^r}{r}-X \ln \Omega \Big) )\eta^{rr} \\
& = \frac{1}{2}( 2\del_r^BX^r
-2\del_u^BX^r )-({\del_u^B} X^u + \del_r^B X^r  )
\big) - \Big( \frac{X^r}{r}-X \ln \Omega \Big) )   \\
& = \frac{1}{2}\del_r^B X^r - \frac{1}{2}{\del_u^B} X^u-  {\del_u^B} X^r
- \Big( \frac{X^r}{r}-X \ln \Omega \Big) ).
\end{align*}
Similarly, for $\Acal^{ur}$ we get
\begin{align*}
\Acal^{ur}  
& =   \frac{1}{2}\big(
\del_\gamma^B(X^u) \eta^{\gamma r}
+ \del_\gamma^B(X^r) \eta^{\gamma u}
-X( \eta^{ur})
-({\del_u^B} X^u + \del_r^B X^r  )\eta^{ur}
\big) - \Big( \frac{X^r}{r}-X \ln \Omega \Big) )\eta^{ur} 
\\
& =   \frac{1}{2} \Big(- \del_u^B X^u+ \del_r^B X^u 
- \del_r^B X^r +{\del_u^B} X^u + \del_r^B X^r  \Big)
+\Big( \frac{X^r}{r}-X \ln \Omega \Big)   
  =  \frac{1}{2}\del_r^B X^u+\Big( \frac{X^r}{r}-X \ln \Omega \Big).
\end{align*}
For $\Acal^{bc}$ we use the hypothesis that our vector field
is radial and get
\begin{align*}
\Acal^{bc}  & =   \frac{1}{2}\Big(
\del_\gamma^B(X^b) \eta^{\gamma c}
+ \del_\gamma^B(X^c) \eta^{\gamma b}
-X( \eta^{bc})
-({\del_u^B} X^u + \del_r^B X^r  \Big)\eta^{bc}
\big) - \Big( \frac{X^r}{r}-X \ln \Omega \Big) )\eta^{bc} \\
& =   \Big(\frac{X^{r}}{r}
- \frac{1}{2}{\del_u^B} X^u- \frac{1}{2}\del_r^B X^r
- \Big( \frac{X^r}{r}-X \ln \Omega \Big) \Big)\frac{\delta^{bc}}{r^{2}}.
\qedhere
\end{align*}
\end{proof}


\begin{proof}[Proof of Lemma~\ref{lem-post}]
By Lemma \ref{vanish_V}, the conformal potential
$V$ identically vanishes for the choices of weight
${}^{I}\Omega$ and ${}^{I\!I}\Omega$. Therefore we have 
$
Q_{\alpha\beta}[\Phi]= \td Q_{\alpha\beta}[\Phi], 
$
where $Q$ is energy-momentum tensor in
Eq. \eqref{em_tensor1}. Expanding
this tensor in polar Bondi coordinates and
using the equation \eqref{metric_bondi1} for
the metric coefficients gives:
\begin{align*}
&Q_{uu}[\Phi] =
(\del^B_{u}\Phi)^{2}+ \frac{1}{2}
(\del^B_{r}\Phi)^{2}
+ \frac{1}{2}|\snabla\Phi|^{2}- \del^B_{u}\Phi
\del^B_{r}\Phi,
 \\ 
&Q_{ur}[\Phi] =
\frac{1}{2}
(\del^B_{r}\Phi)^{2}
+ \frac{1}{2}|\snabla\Phi|^{2}, \qquad
Q_{rr}[\Phi] =
(\del^B_{r}\Phi)^{2}.
\end{align*}
Since the rescaled normal
is $N{'} =(1-{r \over t} )\del^B_{u}+ {r \over t} \del^B_{r}$
we use these identities to get:
\begin{align*}
{}^{(X)}\!\td P_{\alpha} N{'}^\alpha  
& =X^{u}(1-{r \over t} )Q_{uu}[\Phi]+X^{u}{r \over t} Q_{ur}[\Phi]
 +X^{r}(1-{r \over t} )Q_{ru}[\Phi]+X^{r}{r \over t} Q_{rr}[\Phi]
 + \frac{\Lambda}{\Omega^2} \Bcal\Phi^2\\
 & =X^{u}\big[(1-{r \over t} )\big((\del^B_{u}\Phi)^2+
\frac{1}{2}(\del^B_{r}\Phi)^2+ \frac{1}{2}|\snabla\Phi|^2
- \del^B_u\Phi\del_r^B\Phi\big)
+{r \over t} \big(\frac{1}{2}(\del^B_{r}\Phi)^{2}
+ \frac{1}{2}|\snabla\Phi|^{2} \big)\big] \\
& \qquad \qquad + X^{r}\big[(1-{r \over t} )(\frac{1}{2}
(\del^B_{r}\Phi)^{2}
+ \frac{1}{2}|\snabla\Phi|^{2})+ {r \over t} (\del^B_{r}\Phi)^{2}
\big]+ \frac{\Lambda}{\Omega^2} \Bcal\Phi^2.
\end{align*}
Combining terms yields the equation \eqref{form_bdry} immediately.
To prove a coercive lower bound we apply Young inequality
with $\eps$ to the unsigned term
$- \del^B_u\Phi\del_r^B\Phi$
in the second equation. This gives us:
\begin{align}
{}^{(X)}\! \td P_{\alpha} N{'}^\alpha  & \geq 
X^{u}(1-{r \over t} )
\big[(1- \frac{1}{2\eps^{2}})(\del^B_{u}\Phi)^2+
(\frac{1}{2}- \frac{1}{2}\eps^{2})(\del^B_{r}\Phi)^2\big]
+ \frac{1}{2}X^{u}({r \over t} (\del^B_{r}\Phi)^2+|\snabla\Phi|^2)
\label{coercive_bdry2} \\
& \qquad\qquad  + \frac{1}{2}X^{r}\big[(1+ {r \over t} )
(\del^B_{r}\Phi)^{2}
+(1-{r \over t} )|\snabla\Phi|^{2})
\big]+ \frac{\Lambda}{\Omega^2} \Bcal\Phi^2 \notag.
\end{align}
Choosing $\eps$ satisfying the property $\frac{1}{2}<\eps^{2}<1$
and using the positivity of $X^{u}$
yields the coercive bound
\begin{align*}
{}^{(X)}\! \td P_{\alpha} N{'}^\alpha 
\gtrsim_{\eps} X^{u}(1-{r \over t} )
(\del^B_{u}\Phi)^2
+ \big(X^{u}+X^{r}(1+ {r \over t} )\big)(\del^B_{r}\Phi)^2
+ \big(X^{u}+X^{r}(1-{r \over t} )\big)|\snabla\Phi|^2 
+ \frac{\Lambda}{\Omega^2} \Bcal\Phi^2, 
\end{align*}
where the implicit constant depends only on $\eps$.
The upper bound follows trivially from applying
Young inequality to the right-hand side of the equation \eqref{form_bdry}.
This completes
the derivation of \eqref{coercive_bdry1}.
\end{proof}


\begin{thebibliography}{99}

\bibitem{Bachelot88} 
{\sc A. Bachelot,}
Probl\`eme de Cauchy global pour des syst\`emes de Dirac-Klein-Gordon,
Ann. Inst. Henri Poincar\'e 48 (1988), 387--422.

\bibitem{Bigorgne2} 
{\sc L. Bigorgne, D. Fajman, J. Joudioux, J. Smulevici, and M. Thaller,}
Asymptotic stability of Minkowski spacetime with non-compactly supported massless Vlasov matter, 
Arch. Ration. Mech. Anal. 242 (2021), 1--147. 

\bibitem{DIPP}
{\sc S. Dong, P.G. LeFloch, and Z. Wyatt,}
Global evolution of the U(1) Higgs Boson: nonlinear stability and uniform energy bounds, 
Annals Henri Poincar\'e 22 (2021), 677--713. 

\bibitem{DLLei}
{\sc S. Dong, P.G. LeFloch, and Z. Lei,}
The top-order energy of quasilinear wave equations in two space dimensions is uniformly bounded, 
Fundamental Research (2022). 

\bibitem{FJS3} 
{\sc D. Fajman, J. Joudioux, and J. Smulevici,}
The stability of the Minkowski space for the Einstein-Vlasov system,
Anal. PDE 14 (2021), 425--531. 

\bibitem{Georgiev} 
{\sc V. Georgiev,}
Decay estimates for the Klein-Gordon equation,
Comm. Partial Differential Equa. 17 (1992), 1111--1139.

\bibitem{HintzVasy2}
{\sc P. Hintz and A. Vasy,}  
Stability of Minkowski space and polyhomogeneity of the metric,
Ann. PDE 6 (2020), no. 1, Paper No. 2, 146 pp. 
 

\bibitem{Hormander} 
{\sc L. H\"ormander,}
{\sl Lectures on nonlinear hyperbolic differential equations,}
Springer Verlag, Berlin, 1997.

\bibitem{HuneauStingo}
{\sc C. Huneau and A. Stingo,}
Global well-posedness for a system of quasilinear wave equations on a product space, 
Preprint ArXiv:2110.13982. 

\bibitem{IfrimStingo}
{\sc M. Ifrim and A. Stingo,} 
Almost global well-posedness for quasilinear strongly coupled wave-Klein-Gordon systems in two space
dimensions, Preprint ArXiv:1910.12673.

\bibitem{IonescuPausader0} 
{\sc A.D. Ionescu and B. Pausader,} 
Global solutions of quasi-linear systems of Klein-Gordon equations in 3D, 
J. Eur. Math. Soc. 16 (2015), 2355--2431.

\bibitem{IonescuPausader1} 
{\sc A.D. Ionescu and B. Pausader,} 
On the global regularity for a wave-Klein-Gordon coupled system, 
Acta Math. Sin. 35 (2019), 933--986. 

\bibitem{IP3} 
{\sc A.D. Ionescu and B. Pausader,} 
The Einstein-Klein-Gordon coupled system: global stability of the Minkowski solution,  
Princeton University Press, Princeton, NJ, 2021. 

\bibitem{Katayama12a}
{\sc S. Katayama,}
Global existence for coupled systems of nonlinear wave and Klein-Gordon equations in three space dimensions,
Math. Z. 270 (2012), 487--513.

\bibitem{Katayama12b}
{\sc S. Katayama,}
Asymptotic pointwise behavior for systems of semi-linear wave equations in three space dimensions,
J. Hyperbolic Differ. Equ. 9 (2012), 263--323.

\bibitem{KauffmanLindblad} 
{\sc C. Kauffman and H. Lindblad,}
Global stability of Minkowski space for the Einstein-Maxwell-Klein-Gordon system in generalized wave
coordinates, Preprint ArXiv:2109.03270.

\bibitem{Klainerman85}
{\sc S. Klainerman,}
Global existence of small amplitude solutions to nonlinear Klein-Gordon equations in four spacetime dimensions,
Comm. Pure Appl. Math. 38 (1985), 631--641.

\bibitem{Klainerman87}
{\sc S. Klainerman,}
Remarks on the global Sobolev inequalities in the Minkowski space $\RR^{n+1}$,
Comm. Pure Appl. Math. 40 (1987), 111--117.

\bibitem{KriegerSterbenz}
{\sc J. Krieger, J. Sterbenz, and D. Tataru,}
Global well-posedness for the Maxwell-Klein-Gordon equation in 4+1 dimensions: small energy,
Duke Math. J. 164 (2015), 973--1040. 

\bibitem{PLF-YM-book} {\sc P.G. LeFloch and Y. Ma,}
{\sl The hyperboloidal foliation method,} 
World Scientific Press, Singapore, 2014.

\bibitem{PLF-YM-one}{\sc P.G. LeFloch and Y. Ma},
The global nonlinear stability of Minkowski space for self-gravitating massive fields. The wave-Klein-Gordon model,
Comm. Math. Phys. 346 (2016), 603--665.

\bibitem{PLF-YM-3}  
{\sc P.G. LeFloch and Y. Ma,}
Nonlinear stability of self-gravitating massive fields,  
Preprint ArXiv:171210045.

\bibitem{PLF-YM-1}{\sc P.G. LeFloch and Y. Ma},
{\sl The global nonlinear stability of Minkowski space for self-gravitating massive fields,} 
World Scientific Press, Singapore, 2018. 

\bibitem{PLF-YM-SecondPart}{\sc P.G. LeFloch and Y. Ma}, 
Einstein-Klein-Gordon spacetimes in the harmonic near-Minkowski regime, 
Portugal. Math. (2022).  

\bibitem{PLF-YM-Third--Part}{\sc P.G. LeFloch and Y. Ma}, 
Nonlinear stability of self-gravitating massive fields: a wave-Klein Gordon-model, 
Preprint ArXiv, 2022. 

\bibitem{PLF-CW} 
{\sc P.G. LeFloch and C.-H. Wei,}
Boundedness of the total energy of relativistic membranes evolving in a curved spacetime,  
{J. Differential Equations }265 (2018), 312--331. 

\bibitem{LR2} {\sc H. Lindblad and I. Rodnianski,}
The global stability of Minkowski spacetime in harmonic gauge,
Ann. of Math. 171 (2010), 1401--1477.

\bibitem{LSoffer} {\sc H. Lindblad and A. Soffer,}
A remark on asymptotic completeness for the critical nonlinear Klein-Gordon equation, 
Letters in Mathematical Physics 73 (2005), 249--258. 

\bibitem{LSogge}
{\sc H. Lindblad and C.D. Sogge,} 
Restriction theorems and semilinear Klein-Gordon equations in (1 + 3)-dimensions, 
Duke Math. Jour.  85 (1996), 227--252. 

\bibitem{LindbladSterbenz}
{\sc H. Lindblad and J. Sterbenz,} 
Global stability for charged-scalar fields on Minkowski space, 
IMRP Int. Math. Res. Pap. 109  (2006), 52976.

\bibitem{LTay} {\sc H. Lindblad and M. Taylor,}  
Global stability of Minkowski space for the Einstein--Vlasov system in the harmonic gauge, 
Preprint ArXiv:1707.06079.

 

\bibitem{OhTataru} 
{\sc S.-J. Oh and D.Tataru,}
Energy dispersed solutions for the (4+1)-dimensional Maxwell-Klein-Gordon equation,
Amer. J. Math. 140 (2018), 1--82.

\bibitem{Oliver} 
{\sc J. Oliver,}
A vector field method for non-trapping, radiating spacetimes, 
J. Hyper. Diff. Equa. 13 (2016), 735--790.

\bibitem{OliverSterbenz} 
{\sc J. Oliver and J. Sterbenz,}
A vector field method for radiating black hole spacetimes,
 Anal. PDE 13 (2020), 29--92.   

\bibitem{OzawaTsutayaTsutsumi}
{\sc T. Ozawa, K. Tsutaya, and Y. Tsutsumi,}
Well-posedness in energy space for the Cauchy problem of the Klein-Gordon-Zakharov equations with different propagation speeds in three space dimensions,
Math. Ann. 313 (1999), 127--140. 

\bibitem{Tataru1} {\sc D. Tataru,} 
Strichartz estimates in the hyperbolic space and global existence for the semilinear wave equation,
Trans. Amer. Math. Soc. 353 (2001), 795--807.  

\bibitem{Tsutsumi-1}
{\sc Y. Tsutsumi,}
Stability of constant equilibrium for the Maxwell-Higgs equations,
 Funkcial. Ekvac. 46 (2003), 41-62.
 
\bibitem{Tsutsumi-2}
{\sc Y. Tsutsumi,}
Global solutions for the Dirac-Proca equations with small initial data in 3+1 space time dimensions,
J. Math. Anal. Appl. 278 (2003), 485--499. 

\bibitem{Wang1} {\sc Q. Wang}, 
An intrinsic hyperboloid approach for Einstein Klein-Gordon equations,  
J. Differential Geom. 115 (2020), 27--109.   
 
\end{thebibliography}
\end{document}